\documentclass{siamltex}%
\usepackage{amsfonts}
\usepackage{amsmath}
\usepackage{amssymb}
\usepackage{graphicx}
\usepackage{enumerate}
\usepackage{color}
\usepackage{hyperref}%
\newtheorem{example}[theorem]{Example}

\newtheorem{remark}[theorem]{Remark}

\begin{document}

\title{Control Sets for Affine Systems, Spectral Properties and Projective Spaces}
\author{Fritz Colonius\\Institut f\"{u}r Mathematik, Universit\"{a}t Augsburg, Augsburg, Germany
\and Alexandre J. Santana and Juliana Setti\\Departamento de Matem\'{a}tica, Universidade Estadual de Maring\'{a}\\Maring\'{a}, Brazil}
\maketitle

\textbf{Abstract: }For affine control systems with bounded control range the
control sets, i.e., the maximal subsets of complete approximate
controllability, are studied using spectral properties. For hyperbolic systems
there is a unique control set with nonvoid interior and it is bounded. For
nonhyperbolic systems, these control sets are unbounded. In an appropriate
compactification of the state space there is a unique chain control set, and
the relations to the homogenous part of the control system are worked out.



\section{Introduction}

We study controllability properties for affine control systems of the form
\begin{equation}
\dot{x}(t)=Ax(t)+\sum_{i=1}^{m}u_{i}(t)(B_{i}x(t)+c_{i})+d,\quad u(t)\in
\Omega, \label{affine}%
\end{equation}
where $A,B_{1},\ldots,B_{m}\in\mathbb{R}^{n\times n}$ and $c_{1},\ldots
,c_{m},d\in\mathbb{R}^{n}$. The controls $u=(u_{1},\ldots,u_{m})$ have values
in a bounded set $\Omega\subset\mathbb{R}^{m}$ with $0\in\Omega$. The set of
admissible controls is $\mathcal{U}=\{u\in L^{\infty}(\mathbb{R}%
,\mathbb{R}^{m})\left\vert u(t)\in\Omega\text{ for almost all }t\right.  \}$
or the set $\mathcal{U}_{pc}$ of all piecewise constant functions defined on
$\mathbb{R}$ with values in $\Omega$. We also write (\ref{affine}) as%
\[
\dot{x}(t)=A(u(t))x(t)+Cu(t)+d,\quad u(t)\in\Omega,
\]
with $A(u):=A+\sum_{i=1}^{m}u_{i}B_{i}$ for $u\in\Omega$ and $C:=(c_{1}%
,\ldots,c_{m})$. With the vector fields $f_{0}(x)=Ax+d$ and $f_{i}%
(x)=B_{i}x+c_{i}$ on $\mathbb{R}^{n}$ we assume throughout that the following
accessibility (or Lie algebra) rank condition holds,%
\begin{equation}
\dim\mathcal{LA}(f_{0},f_{1},\ldots,f_{m})(x)=n\text{ for all }x\in
\mathbb{R}^{n}, \label{LARC}%
\end{equation}
where $\mathcal{LA}(f_{0},f_{1},\ldots,f_{m})$ is the set of vector fields in
the Lie algebra generated by $f_{0},f_{1},\ldots,f_{m}$.

Controllability properties of bilinear and affine systems have been studied
since more than fifty years. Early contributions are due to Rink and Mohler
\cite{RinkM68} who took the set of equilibria as a starting point for
establishing results on complete controllability. On the other hand,
Lie-algebraic methods have yielded important insights. A classical result due
to Jurdjevic and Sallet \cite[Theorem 2]{JurS84} (cf. also Do Rocio, Santana,
and Verdi \cite{DSC09}) shows that affine system (\ref{affine}) is
controllable on $\mathbb{R}^{n}$ if it has no fixed points and its homogeneous
part, the bilinear system
\begin{equation}
\dot{x}(t)=A(u(t))x(t),\quad u(t)\in\Omega, \label{hom_n}%
\end{equation}
is controllable on $\mathbb{R}^{n}\setminus\{0\}$ (only the first condition is
necessary). Initially, bilinear control systems were considered as
\textquotedblleft nearly linear\textquotedblright\ (cf. Bruni, Di Pillo, and
Koch \cite{BDK74} containing also many early references). However it has
turned out that characterizing controllability of such systems (even with
unrestricted controls) is a very difficult problem. As Jurdjevic \cite[p.
182]{Jurd97} emphasizes, the controllability properties of affine systems are
substantially richer and may require \textquotedblleft entirely different
geometrical considerations\textquotedblright\ (based on Lie-algebraic
methods). There is a substantial literature on controllability properties of
bilinear and affine control systems. Here we only refer to the monographs
Mohler \cite{Mohler}, Elliott \cite{Elliott}, and Jurdjevic \cite{Jurd97}.
Further references are also included in \cite{ColRS} where we analyze
controllability properties near equilibria.

The present paper concentrates on control sets, i.e., maximal subsets of
approximate controllability; cf. Definition \ref{Definition_control_sets}
(another relaxation of complete controllability is the notion of
\textquotedblleft near controllability\textquotedblright\ studied by Nie
\cite{Nie20}\ for discrete-time bilinear control systems). The key to our
analysis are properties of the interior of the system semigroup and spectral
properties of the homogeneous part (\ref{hom_n}). In the hyperbolic case (cf.
Definition \ref{Definition_hyperbolic}) Theorem \ref{Theorem_hyperbolic1}
shows that an affine control system has a unique control set $D$ with nonvoid
interior. In the uniformly hyperbolic case (cf. Definition
\ref{Definition_uni_hyperbolic}) Theorem \ref{Theorem_hyperbolic2} shows that
$D$ is bounded. Hence these systems enjoy similar controllability properties
as linear control systems of the form $\dot{x}=Ax+Bu$ with hyperbolic matrix
$A$; cf. Colonius and Kliemann \cite[Example 3.2.16]{ColK00}. We remark that a
generalization in another direction is given for control sets of linear
control systems on Lie groups by Ayala and Da Silva \cite{AyaDaS17}.

Nonhyperbolic affine systems may possess several control sets with nonvoid
interior. By Theorem \ref{Theorem_main2} each of them is unbounded. The proof
takes up ideas from Rink and Mohler \cite{RinkM68}, replacing the set of
equilibria by the set of periodic solutions. Then we compactify the state
space using an embedding into projective space $\mathbb{P}^{n}$. This is in
line with the study of the behavior at infinity for ordinary differential
equations based on the Poincar\'{e} sphere; cf. Perko \cite[Section
3.10]{Perko}. Already the special case of linear control systems discussed in
the beginning of Section \ref{Section7} shows that here chain transitivity (a
classical notion in the theory of dynamical systems; cf. Robinson
\cite{Robin98}) plays an important role. Theorem \ref{Theorem_main3} shows
that for a control set with nonvoid interior the \textquotedblleft boundary at
infinity\textquotedblright\ (cf. Definition \ref{Definition_infinity})
intersects a chain control set of the projectivized homogeneous part (here
small jumps in the trajectories are allowed; cf. Definition \ref{intro2:ccs}).
The main result on the nonhyperbolic case is Theorem \ref{Theorem_main5}
showing that there is a single chain control set in $\mathbb{P}^{n}$
containing the images of all control sets $D$ with nonvoid interior in
$\mathbb{R}^{n}$. The boundary at infinity of this chain control set contains
all chain control sets of the projectivized homogeneous part having nonvoid
intersection with the boundary at infinity of one of the control sets $D$.
These results cast new light on the relations between affine systems and their
homogeneous parts and are intuitively appealing, since one may expect that for
unbounded $x$-values the inhomogeneous part $Cu(t)+d$ becomes less relevant.

The contents of the present paper are as follows. In Section \ref{Section2} we
first recall some notation and general properties of nonlinear control systems
and cite results on spectral properties and controllability for homogeneous
bilinear control systems.\ A result on periodic solutions of periodic
inhomogeneous differential equations is stated. Its proof is given in the
appendix. Section \ref{Section4} analyzes the system semigroup of affine
control systems and their homogeneous parts. Section \ref{Section5} shows that
for hyperbolic systems a unique control set exists and that it is bounded.
Section \ref{Section6} prepares the discussion of nonhyperbolic systems by
embedding affine control systems into homogeneous bilinear systems and
associated systems in projective spaces. Section \ref{Section7} describes the
control sets and their boundaries at infinity for nonhyperbolic
systems.\smallskip

\textbf{Notation.} The set of eigenvalues of a matrix $A\in\mathbb{R}^{n\times
n}$ is $\mathrm{spec}(A)$ and the real eigenspace for an eigenvalue $\mu
\in\mathbb{C}$ is $\mathbf{E}(A;\mu)$. In a metric space $X$ with distance $d$
the distance of $x\in X$ to a nonvoid subset $A\subset X$ is $d(x,A)=\inf
\{d(x,a)\left\vert a\in A\right.  \}$.

\section{Preliminaries\label{Section2}}

In Subsection \ref{Subsection2.1} notation and some basic properties of
control systems are recalled. Subsection \ref{Subsection2.2} cites results on
control sets for homogeneous bilinear control systems and Subsection
\ref{Subsection2.3} characterizes periodic solutions of inhomogeneous periodic
linear differential equations.

\subsection{Basic properties of nonlinear control systems\label{Subsection2.1}%
}

In this subsection we introduce some terminology and notation for
control-affine systems including control sets and chain control sets.

We will consider control-affine systems on a smooth (real analytic) manifold
$M$ of the form%
\begin{equation}
\dot{x}(t)=f_{0}(x(t))+\sum\nolimits_{i=1}^{m}u_{i}(t)f_{i}(x(t)),\quad
u\in\mathcal{U}\text{ or }u\in\mathcal{U}_{pc}, \label{2.1}%
\end{equation}
where $f_{0},f_{1},\ldots,f_{m}$ are smooth vector fields on $M$ and the
control range $\Omega\subset\mathbb{R}^{m}$ is bounded with $0\in\Omega$. We
assume that for every initial state $x\in M$ and every control function
$u\in\mathcal{U}$ there exists a unique solution $\varphi(t,x,u),t\in
\mathbb{R}$, with $\varphi(0,x,u)=x$ of (\ref{2.1}) depending continuously on
$x$. For the general theory of nonlinear control systems we refer to Sontag
\cite{Son98} and Jurdjevic \cite{Jurd97}.

The set of points reachable from $x\in M$ and controllable to $x\in M$ up to
time $T>0$ are defined by
\begin{align*}
{\mathcal{O}}_{\leq T}^{+}(x)  &  :=\{y\in M\left\vert \text{there are}\;0\leq
t\leq T\;\text{and}\;u\in\mathcal{U}\;\text{with}\;y=\varphi(t,x,u)\right.
\},\\
{\mathcal{O}}_{\leq T}^{-}(x)  &  :=\{y\in M\left\vert \text{there are}\;0\leq
t\leq T\;\text{and}\;u\in\mathcal{U}\;\text{with}\;x=\varphi(t,y,u)\right.
\},
\end{align*}
resp. Furthermore, the reachable set (or \textquotedblleft positive
orbit\textquotedblright) from $x$ and the set controllable to $x$ (or
\textquotedblleft negative orbit\textquotedblright\ of $x$) are%
\[
\mathcal{O}^{+}(x)=\bigcup\nolimits_{T>0}O_{\leq T}^{+}(x),\quad
\mathcal{O}^{-}(x)=\bigcup\nolimits_{T>0}O_{\leq T}^{-}(x),
\]
resp. The system is called locally accessible in $x$ if $\mathcal{O}_{\leq
T}^{+}(x)$ and $\mathcal{O}_{\leq T}^{-}(x)$ have nonvoid interior for all
$T>0$, and the system is called locally accessible if this holds in every
point $x\in M$. This is equivalent to the following accessibility rank
condition%
\begin{equation}
\dim\mathcal{LA}\left\{  f_{0},f_{1},\ldots,f_{m}\right\}  (x)=\dim M\text{
for all }x\in M; \label{ARC}%
\end{equation}
here $\mathcal{LA}\left\{  f_{0},f_{1},\ldots,f_{m}\right\}  (x)$ is the
subspace of the tangent space $T_{x}M$ corresponding to the vector fields,
evaluated in $x$, in the Lie algebra generated by $f_{0},f_{1},\ldots,f_{m}$.
The sets $\mathcal{O}^{-}(x)$ are the reachable sets of the time reversed
system given by $\dot{x}(t)=-f_{0}(x(t))-\sum\nolimits_{i=1}^{m}u_{i}%
(t)f_{i}(x(t))$. The trajectories for controls in $\mathcal{U}$ can be
uniformly approximated on bounded intervals by trajectories for controls in
$\mathcal{U}_{pc}$.

The following definition introduces sets of complete approximate controllability.

\begin{definition}
\label{Definition_control_sets}A nonvoid set $D\subset M$ is called a control
set of system (\ref{2.1}) if it has the following properties: (i) for all
$x\in D$ there is a control function $u\in\mathcal{U}$ such that
$\varphi(t,x,u)\in D$ for all $t\geq0$, (ii) for all $x\in D$ one has
$D\subset\overline{\mathcal{O}^{+}(x)}$, and (iii) $D$ is maximal with these
properties, that is, if $D^{\prime}\supset D$ satisfies conditions (i) and
(ii), then $D^{\prime}=D$. A control set $D\subset M$ is called an invariant
control set if $\overline{D}=\overline{\mathcal{O}^{+}(x)}$ for all $x\in D$.
All other control sets are called variant.
\end{definition}

We recall some properties of control sets; cf. Colonius and Kliemann
\cite[Chp. 3]{ColK00}.

\begin{remark}
\label{Remark2.2}If the intersection of two control sets is nonvoid, the
maximality property (iii) implies that they coincide. If the system is locally
accessible in all $x\in\overline{D}$, then by \cite[Lemma 3.2.13(i)]{ColK00}
$\overline{D}=\overline{\mathrm{int}(D)}$ and $D=\mathcal{O}^{-}%
(x)\cap\overline{\mathcal{O}^{+}(x)}$ for all $x\in\mathrm{int}(D)$ and
$\mathrm{int}(D)\subset\mathcal{O}^{+}(x)$ for all $x\in D$ (here it suffices
to consider controls in $\mathcal{U}_{pc}$). If local accessibility holds on
$M$, the control sets with nonvoid interior for controls in $\mathcal{U}$
coincide with those defined analogously for controls in $\mathcal{U}_{pc}$.
This is proved using the approximation by trajectories for controls in
$\mathcal{U}_{pc}$.
\end{remark}

\begin{lemma}
\label{Lemma_intersection}Suppose that the system is locally accessible and
there is $x\in\mathrm{int}(\mathcal{O}^{-}(x))\cap\mathrm{int}(\mathcal{O}%
^{+}(x))$. Then $D:=\mathcal{O}^{-}(x)\cap\overline{\mathcal{O}^{+}(x)}$ is a
control set and $x\in\mathrm{int}(D)$.
\end{lemma}

\begin{proof}
\marginpar{10.5.}The set $\mathrm{int}(\mathcal{O}^{-}(x))\cap\mathrm{int}%
(\mathcal{O}^{+}(x))$ satisfies properties (i) and (ii) of control sets, hence
it is contained in a control set $D^{\prime}$ and $x\in\mathrm{int}(D^{\prime
})$. Then $D^{\prime}=\mathcal{O}^{-}(x)\cap\overline{\mathcal{O}^{+}(x)}$ by
Remark \ref{Remark2.2}.
\end{proof}

Next we introduce a notion of controllability allowing for (small) jumps
between pieces of trajectories. Here we fix a metric $d$ on $M$.

\begin{definition}
\label{intro2:defchains}Fix $x,y\in M$ and let $\varepsilon,T>0.$ A controlled
$(\varepsilon,T)$\textit{-chain} $\zeta$ from $x$ to $y$ is given by
$n\in\mathbb{N},\ x_{0}=x,\ldots,x_{n-1},x_{n}=y\in M,\ u_{0},\ldots
,u_{n-1}\in\mathcal{U}$, and $t_{0},\ldots,t_{n-1}\geq T$ with
\[
d(\varphi(t_{j},x_{j},u_{j}),x_{j+1})\leq\varepsilon\text{ }%
\,\text{for\thinspace all}\,\,\,j=0,\ldots,n-1.
\]
If for every $\varepsilon,T>0$ there is a controlled $(\varepsilon,T)$-chain
from $x$ to $y$, then the point $x$ is chain controllable to $y.$
\end{definition}

In analogy to control sets, chain control sets are defined as maximal regions
of chain controllability; cf. \cite[Chapter 4]{ColK00}.

\begin{definition}
\label{intro2:ccs}A nonvoid set $E\subset M$ is called a \textit{chain control
set} of system (\ref{2.1}) if (i) for all $x\in E$ there is $u\in\mathcal{U}$
such that $\varphi$$(t,x,u)\in E$ for all $t\in\mathbb{R}$, (ii) for all
$x,y\in E$ and $\varepsilon,\,T>0$ there is a controlled $(\varepsilon
,T)$-chain from $x$ to $y$, and (iii) $E$ is maximal (with respect to set
inclusion) with these properties.
\end{definition}

Obviously, every equilibrium and every periodic trajectory is contained in a
control set and a chain control set. Since the concatenation of two controlled
$(\varepsilon,T)$-chains again yields a controlled $(\varepsilon,T)$-chain,
two chain control sets coincide if their intersection is nonvoid.

For a continuous dynamical system $\psi:\mathbb{R}\times X\rightarrow X$ on a
metric space $X$ a subset $Y\subset X$ is called chain transitive, if for all
$x,y\in Y$ and all $\varepsilon,T>0$ there is an $(\varepsilon,T)$-chain from
$x$ to $y$ given by $n\in\mathbb{N},\,x_{0}=x,x_{1},\ldots,x_{n}=y\in
X$,$\mathcal{\ }$and $t_{0},\ldots,t_{n-1}\geq T$ with $d(\psi(t_{j}%
,x_{j}),x_{j+1})\leq\varepsilon$ $\,$for\thinspace all$\,\,\,j=0,\ldots,n-1$.

For compact and convex control range $\Omega$, a control system of the form
(\ref{2.1}) defines a continuous dynamical system, the control flow, given by
$\mathbf{\Psi}:\mathbb{R}\times\mathcal{U}\times M\rightarrow\mathcal{U}\times
M,\,(t,u,x)\mapsto(u(t+\cdot),\varphi(t,x,u))$, where $u(t+\cdot
)(s):=u(t+s),\,s\in\mathbb{R}$, and $\mathcal{U}\subset L^{\infty}%
(\mathbb{R},\mathbb{R}^{m})$ considered in a metric for the weak$^{\ast}$
topology is compact; cf. Kawan \cite[Proposition 1.17]{Kawa13}. The following
assertions are shown in \cite[Proposition 1.24]{Kawa13}: chain control sets
are closed and for locally accessible systems every control set with nonvoid
interior is contained in a chain control set. The chain control sets $E$
uniquely correspond to the maximal invariant chain transitive sets
$\mathcal{E}$ of the control flow $\mathbf{\Psi}$ via%
\begin{equation}
\mathcal{E}:=\{(u,x)\in\mathcal{U}\times E\left\vert \varphi(t,x,u)\in E\text{
for all }t\in\mathbb{R}\right.  \}. \label{chain_transitive1}%
\end{equation}

\subsection{Control sets for homogeneous bilinear systems\label{Subsection2.2}%
}

In this subsection we cite several results on control sets and chain control
sets for homogeneous bilinear control systems of the form%
\begin{equation}
\dot{x}(t)=A(u(t))x(t),\quad u(t)\in\Omega,\text{ with }A(u):=A+\sum_{i=1}%
^{m}u_{i}B_{i},u\in\Omega, \label{hom}%
\end{equation}
where $A,B_{1},\ldots,B_{m}\in\mathbb{R}^{n\times n}$ and the controls
$u=(u_{1},\ldots,u_{m})$ have values in a compact convex neighborhood $\Omega$
of the origin in $\mathbb{R}^{m}$. The solutions are denoted by $\varphi
_{\hom}(t,x,u),t\in\mathbb{R}$.

By homogeneity a system of the form (\ref{hom}) induces control systems on the
unit sphere $\mathbb{S}^{n-1}$ and on projective space $\mathbb{P}^{n-1}$. The
projections of $\mathbb{R}^{n}\setminus\{0\}$ to the unit sphere
$\mathbb{S}^{n-1}$ and projective space $\mathbb{P}^{n-1}$ are denoted by
$\pi_{\mathbb{S}}$ and $\pi_{\mathbb{P}}$, resp.

\begin{definition}
\label{Definition_spectra_bilinear}Let $\lambda(u,x)=\lim\sup_{t\rightarrow
\infty}\frac{1}{t}\log\left\Vert \varphi_{\hom}(t,x,u)\right\Vert $ be the
Lyapunov exponent for $(u,x)\in\mathcal{U}\times\left(  \mathbb{R}%
^{n}\setminus\{0\}\right)  $.

(i) The Floquet spectrum of a control set $_{\mathbb{S}}D$ on the unit sphere
$\mathbb{S}^{n-1}$ is
\[
\Sigma_{Fl}(_{\mathbb{S}}D)=\left\{  \lambda(u,x)\left\vert \pi_{\mathbb{S}%
}x\in\mathrm{int}(_{\mathbb{S}}D),\,u\in\mathcal{U}_{pc}\text{ }%
\tau\text{-periodic with }\pi_{\mathbb{S}}\varphi_{\hom}(\tau,x,u)=\pi
_{\mathbb{S}}x\right.  \right\}  .
\]

(ii) The Floquet spectrum of a control set $_{\mathbb{P}}D$ on projective
space $\mathbb{P}^{n-1}$ is
\[
\Sigma_{Fl}(_{\mathbb{P}}D)=\left\{  \lambda(u,x)\left\vert \pi_{\mathbb{P}%
}x\in\mathrm{int}(_{\mathbb{P}}D),\,u\in\mathcal{U}_{pc}\text{ }%
\tau\text{-periodic with }\pi_{\mathbb{P}}\varphi_{\hom}(\tau,x,u)=\pi
_{\mathbb{P}}x\right.  \right\}  .
\]

(iii) The Lyapunov spectrum of a control set $_{\mathbb{P}}D$ on projective
space $\mathbb{P}^{n-1}$ is%
\[
\Sigma_{Ly}(_{\mathbb{P}}D)=\left\{  \lambda(u,x)\left\vert u\in
\mathcal{U}\text{ and }\pi_{\mathbb{P}}\varphi(t,x,u)\in\overline
{_{\mathbb{P}}D}\text{ for all }t\geq0\right.  \right\}  .
\]

\end{definition}

\begin{remark}
\label{Remark_semi}The Floquet spectrum can be characterized using the system
semigroups $_{\mathbb{R}}\mathcal{S}^{\hom}$ and $_{\mathbb{P}}\mathcal{S}%
^{\hom}$ of the systems on $\mathbb{R}^{n}\setminus\{0\}$ and on
$\mathbb{P}^{n-1}$, resp. (cf. Section \ref{Section4} for the definition of
system semigroups). Suppose that the accessibility rank condition in
$\mathbb{P}^{n-1}$ holds. Corollary 7.3.18 in Colonius and Kliemann
\cite{ColK00} implies that the Floquet spectrum of a control set
$_{\mathbb{P}}D$ consists of the numbers $\frac{1}{\tau}\log\left\vert
\rho\right\vert $ where $\rho$ is a real eigenvalue of an element $\Phi
_{u}(\tau,0)\in\,_{\mathbb{R}}\mathcal{S}_{\tau}^{\hom}$ with eigenspace
$\pi_{\mathbb{P}}\left(  \mathbf{E}(\Phi_{u}(\tau,0);\rho)\right)
\subset\mathrm{int}(_{\mathbb{P}}D)$ and such that $\Phi_{u}(\tau,0)$ induces
an element of the system semigroup in $\mathrm{int}(_{\mathbb{P}}%
\mathcal{S}_{\leq\tau+1}^{\hom})$.
\end{remark}

The following theorem analyzes the control sets of the homogeneous system
(\ref{hom}).

\begin{theorem}
\label{Theorem_95}Consider the systems on the unit sphere $\mathbb{S}^{n-1}$
and on projective space $\mathbb{P}^{n-1}$ obtained by projection of the
homogeneous bilinear control system (\ref{hom}). Assume that the accessibility
rank condition on $\mathbb{P}^{n-1}$ is satisfied.

(i) There are $1\leq k_{0}\leq n$ control sets $_{\mathbb{P}}D_{j}$ with
nonvoid interior in $\mathbb{P}^{n-1}$ and exactly one of these control sets
is an invariant control set.

(ii) There are $1\leq k_{1}\leq2k_{0}$ control sets $_{\mathbb{S}}D_{i}$ on
$\mathbb{S}^{n-1}$. In each case, one or two of the control sets
$_{\mathbb{S}}D_{i}$ on $\mathbb{S}^{n-1}$ project to a single control set
$_{\mathbb{P}}D_{j}$ on $\mathbb{P}^{n-1}$ and then $\Sigma_{Fl}(_{\mathbb{S}%
}D_{i})=\Sigma_{Fl}(_{\mathbb{P}}D_{j})$ and $\Sigma_{Ly}(_{\mathbb{S}}%
D_{i})=\Sigma_{Ly}(_{\mathbb{P}}D_{j})$.

(iii) If $0$ is in the interior of the Floquet spectrum $\Sigma_{Fl}%
(_{\mathbb{S}}D_{i})$ the cone%
\[
_{\mathbb{R}}D_{i}:=\{\alpha x\in\mathbb{R}^{n}\left\vert \alpha>0\text{ and
}x\in\,_{\mathbb{S}}D_{i}\right.  \}
\]
generated by $_{\mathbb{S}}D_{i}$ is a control set with nonvoid interior in
$\mathbb{R}^{n}\setminus\{0\}$. Conversely, if $_{\mathbb{R}}D$ is control set
with nonvoid interior in $\mathbb{R}^{n}\setminus\{0\}$, then $0\in
\overline{\Sigma_{Ly}(_{\mathbb{P}}D)}$, where $_{\mathbb{P}}D\supset
\pi_{\mathbb{P}}(_{\mathbb{R}}D)$.
\end{theorem}

\begin{proof}
For (i) see \cite[Theorem 7.1.1]{ColK00}. Assertion (ii) follows from
Colonius, Santana, and Setti \cite[Theorem 3.15]{ColRS} and the observation
that the equality of the Lyapunov spectra is obvious. (iii) follows from
\cite[Proposition 3.18]{ColRS} and (ii).
\end{proof}

We turn to the chain control sets in projective space. By \cite[Theorem
7.1.2]{ColK00} every chain control set $_{\mathbb{P}}E_{j}$ contains a control
set $_{\mathbb{P}}D_{i}$ with nonvoid interior, hence the number $l$ of chain
control sets satisfies $1\leq l\leq k_{0}$. Furthermore, \cite[Theorem
7.3.16]{ColK00} shows that for every chain control set $_{\mathbb{P}}E_{j}$ in
$\mathbb{P}^{n-1}$ and every $u\in\mathcal{U}$%
\[
\{x\in\mathbb{R}^{n}\left\vert x\not =0\text{ implies }\pi_{\mathbb{P}}%
\varphi_{\hom}(t,x,u)\in\,_{\mathbb{P}}E_{j}\text{ for all }t\in
\mathbb{R}\right.  \}
\]
is a linear subspace and its dimension is independent of $u\in\mathcal{U}$. By
(\ref{chain_transitive1}) the chain control sets $_{\mathbb{P}}E_{j}$ uniquely
correspond to the maximal chain transitive subsets $_{\mathbb{P}}%
\mathcal{E}_{j}$ of the control flow on $\mathcal{U}\times\mathbb{P}^{n-1}$
via%
\begin{equation}
_{\mathbb{P}}\mathcal{E}_{j}:=\{(u,\pi_{\mathbb{P}}x)\in\mathcal{U}%
\times\mathbb{P}^{n-1}\left\vert \pi_{\mathbb{P}}\varphi_{\hom}(t,x,u)\in
\,_{\mathbb{P}}E_{j}\text{ for all }t\in\mathbb{R}\right.  \}.
\label{chain_transitive2}%
\end{equation}

\subsection{Periodic solutions\label{Subsection2.3}}

We state some facts on periodic solutions of inhomogeneous periodic
differential equations of the form%
\begin{equation}
\dot{x}(t)=P(t)x(t)+z(t), \label{periodic2}%
\end{equation}
where $P(\cdot)\in L^{\infty}(\mathbb{R},\mathbb{R}^{n\times n})$ and
$z(\cdot)\in L^{\infty}(\mathbb{R},\mathbb{R}^{n})$ are $\tau$-periodic, i.e.,
$P(t+\tau)=P(t)$ and $z(t+\tau)=z(t)$ for almost all $t\in\mathbb{R}$. The
principal fundamental solution $\Phi(t,s)\in\mathbb{R}^{n\times n}%
,t,s\in\mathbb{R}$, is given by%
\[
\frac{d}{dt}\Phi(t,s)=P(t)\Phi(t,s)\text{ with }\Phi(s,s)=I.
\]
The homogeneous equation with $z(t)\equiv0$ has nontrivial (non-unique) $\tau
$-periodic solutions if and only if $1$ is an eigenvalue of $\Phi(\tau,0)$.
Here and in the following, uniqueness of a periodic solution means that it is
unique up to time shifts.

The Floquet multipliers $\rho_{j}\in\mathbb{C},j=1,\ldots,n$, are defined as
the eigenvalues of $\Phi(\tau,0)$ and the Floquet exponents are $\lambda
_{j}:=\frac{1}{\tau}\log\left\vert \rho_{j}\right\vert $. \marginpar{10.5.}%
They coincide with the Lyapunov exponents; cf. Chicone \cite[Proposition
2.61]{Chic99}, Colonius and Kliemann \cite[Theorem 7.2.9]{ColK14}. In
particular, $0$ is a Floquet exponent if $1$ is a Floquet multiplier. We also
refer to Chicone \cite[Section 2.4]{Chic99} and Teschl \cite[Section 3.6]{Tes}
for background on Floquet theory (note that the Floquet exponents as defined
above are the real parts of the Floquet exponents defined in \cite{Chic99} and
\cite{Tes}).

We need the following results on periodic solutions.

\begin{proposition}
\label{Proposition_periodicODE}Consider the $\tau$-periodic differential
equation (\ref{periodic2}).

(i) There is a unique $\tau$-periodic solution if and only if $1\not \in
\mathrm{spec}(\Phi(\tau,0))$. Its initial value (at time $0$) is
$x^{0}=(I-\Phi(\tau,0))^{-1}\int_{0}^{\tau}\Phi(\tau,s)z(s)ds$.

(ii) If for $z(t)\not \equiv 0$ there does not exist a $\tau$-periodic
solution, then the principal fundamental solution satisfies $1\in
\mathrm{spec}(\Phi(\tau,0))$ and $\int_{0}^{\tau}\Phi(\tau,s)z(s)ds\not \in
\operatorname{Im}(I-\Phi(\tau,0))$.

(iii) For $k=0,1,\ldots$, let $P^{k}(\cdot)$ and $z^{k}(\cdot)$ be $\tau_{k}%
$-periodic and suppose that for a $c>0$ the norms in $L^{\infty}(\left[
0,\tau_{0}+1\right]  ;\mathbb{R}^{n})$ satisfy $\left\Vert P^{k}%
(\cdot)\right\Vert _{\infty},\allowbreak\left\Vert z^{k}(\cdot)\right\Vert
_{\infty}\leq c$ for all $k$. Assume that $\tau_{k}\rightarrow\tau_{0}%
,~P^{k}(\cdot)\rightarrow P^{0}(\cdot)$ in $L^{1}([0,\tau_{0}+1],\mathbb{R}%
^{n\times n})$, and $z^{k}(\cdot)\rightarrow z^{0}(\cdot)$ in $L^{1}(\left[
0,\tau_{0}+1\right]  ;\mathbb{R}^{n})$ for $k\rightarrow\infty$. Then for
$k\rightarrow\infty$ the corresponding principal fundamental matrices
$\Phi^{k}(t,s)$ converge to $\Phi^{0}(t,s)$ uniformly in $t,s\in\lbrack
0,\tau_{0}+1]$.

(iv) In the situation of (iii) assume additionally that for $k=0,1,\ldots$ the
corresponding principal fundamental solutions $\Phi^{k}(t,s)$ satisfy
$1\not \in \mathrm{spec}(\Phi^{k}(\tau_{k},0))$. Then the initial values
$x^{k}$ of the corresponding unique $\tau_{k}$-periodic solutions converge for
$k\rightarrow\infty$ to the initial value $x^{0}$ of the unique $\tau_{0}%
$-periodic solution for $P^{0}(\cdot)$ and $z^{0}(\cdot)$.
\end{proposition}

\begin{proof}
see Appendix.
\end{proof}

\section{System semigroups of affine systems in $\mathbb{R}^{n}$%
\label{Section4}}

In this section we analyze system semigroups for affine systems of the form
(\ref{affine}); cf. also Jurdjevic and Sallet \cite{JurS84}.

We start with the following general remarks on the relevant Lie group which is
the semidirect product $G=\mathbb{R}^{n}\ltimes GL(\mathbb{R}^{n})$ with
product given by%
\[
(v,g)\cdot(w,h)=(v+gw,gh).
\]
Its Lie algebra is given by the semidirect product $\mathfrak{g}%
=\mathbb{R}^{n}\ltimes\mathfrak{gl}(\mathbb{R}^{n})$, where the Lie bracket is%
\[
\lbrack(a,A),(b,B)]=(Ab-Ba,AB-BA).
\]
If we consider $X=(a,A),\,Y=(b,B)$ as vector fields on $\mathbb{R}^{n}$
through the relation $X(x)=Ax+a,\,Y(x)=Bx+b$, the Lie bracket is%
\[
\lbrack X,Y](x)=-(AB-BA)x-(Ab-Ba).
\]
If we denote by $e=(0,I)\in G$ the identity element, then the tangent space is
$T_{e}G=\mathfrak{g}$ and an element $X\in\mathfrak{g}$ can be identified with
a right-invariant vector field, a smooth vector field on $G$, given by
$X^{R}(g):=(dR_{g})_{e}X$, where $R_{g}$ stands for the right-translation and
$(dR_{g})_{e}$ is its differential at the identity element. By standard
results, the vector fields $X^{R}$ are complete and their flows satisfy%
\begin{equation}
\phi_{t}^{X^{R}}(g)=R_{g}(\phi_{t}^{X^{R}}(e))\text{ for all }g\in G.
\label{R}%
\end{equation}
Note also that%
\[
(\exp tX)(x)=e^{tA}x+\int_{0}^{t}e^{(t-s)A}a\,ds\text{ for }x\in\mathbb{R}%
^{n}\text{ and }X=(a,A),
\]
gives us exactly the expression for the exponential map for $G=\mathbb{R}%
^{n}\ltimes GL(\mathbb{R}^{n})$ and $\mathfrak{g}=\mathbb{R}^{n}%
\ltimes\mathfrak{gl}(\mathbb{R}^{n})$, meaning that the Lie group exponential
is given by%
\[
\exp tX=\left(  \int_{0}^{t}e^{(t-s)A}a\,ds,e^{tA}\right)  \text{ for
}X=(a,A).
\]
If $\mathcal{F}\subset\mathfrak{g}$ is a nonempty subset, consider the
subgroups of $G$%
\begin{align*}
\mathcal{G}(\mathcal{F})  &  :=\left\{  \exp t_{1}X_{1}\cdot\exp t_{2}%
X_{2}\cdots\exp t_{k}X_{k}\left\vert t_{i}\in\mathbb{R}\text{ and }X_{i}%
\in\mathcal{F}\right.  \right\}  ,\\
\mathcal{G}^{R}(\mathcal{F})  &  :=\{(\phi_{t_{1}}^{X_{1}^{R}}\circ\phi
_{t_{2}}^{X_{2}^{R}}\circ\cdots\circ\phi_{t_{k}}^{X_{k}^{R}})(e)\left\vert
t_{i}\in\mathbb{R}\text{ and }X_{i}\in\mathcal{F}\right.  \},
\end{align*}
and the semigroups $\mathcal{S}(\mathcal{F}),\,\mathcal{S}^{R}(\mathcal{F})$
where only $t_{i}>0$ are allowed. Thus $\mathcal{S}^{R}(\mathcal{F})$ is the
set of points on $G$ that can be attained from $e\in G$ by concatenations of
the flows of $\mathcal{F}^{R}=\{X^{R}\left\vert X\in\mathcal{F}\right.  \}$.

\begin{proposition}
\label{Proposition_semigroups}If $\mathcal{F}\subset\mathfrak{g}$ is a
nonempty subset, the groups $\mathcal{G}(\mathcal{F})$ and $\mathcal{G}%
^{R}(\mathcal{F})$ as well as the semigroups $\mathcal{S}(\mathcal{F})$ and
$\mathcal{S}^{R}(\mathcal{F})$, resp., coincide.
\end{proposition}

\begin{proof}
Since $X^{R}(e)=X$, the exponential map $\exp:\mathfrak{g}\rightarrow G$ is
defined by $\exp X=\phi_{1}^{X^{R}}(g)$. Consequently we get for all $t_{i}%
\in\mathbb{R}$ and $X_{i}\in\mathcal{F}\subset\mathfrak{g}$, using (\ref{R}),%
\begin{align*}
&  \exp t_{1}X_{1}\cdot\exp t_{2}X_{2}\cdots\exp t_{k}X_{k}\\
&  =R_{e^{t_{1}X_{1}}e^{t_{2}X_{2}}\cdots e^{t_{k-1}X_{k-1}}}(e^{t_{k}X_{k}%
})=R_{e^{t_{1}X_{1}}e^{t_{2}X_{2}}\cdots e^{t_{k-1}X_{k-1}}}(\phi_{t_{k}%
}^{X_{k}^{R}}(e))\\
&  =\phi_{t_{k}}^{X_{k}^{R}}(e^{t_{1}X_{1}}e^{t_{2}X_{2}}\cdots e^{t_{k-1}%
X_{k-1}})=\phi_{t_{k}}^{X_{k}^{R}}\left(  R_{e^{t_{1}X_{1}}\cdots
e^{t_{k-2}X_{k-2}}}(e^{t_{k-1}X_{k-1}})\right) \\
&  =\phi_{t_{k}}^{X_{k}^{R}}\left(  R_{e^{t_{1}X_{1}}e^{t_{2}X_{2}}\cdots
e^{t_{k-2}X_{k-2}}}(\phi_{t_{k-1}}^{X_{k-1}^{R}}(e))\right)  =\phi_{t_{k}%
}^{X_{k}^{R}}\circ\phi_{t_{k-1}}^{X_{k-1}^{R}}(e^{t_{1}X_{1}}\cdots
e^{t_{k-2}X_{k-2}})\\
&  =\,\cdots\,=\left(  \phi_{t_{k}}^{X_{k}^{R}}\circ\phi_{t_{k-1}}%
^{X_{k-1}^{R}}\circ\cdots\circ\phi_{t_{1}}^{X_{1}^{R}}\right)  (e).
\end{align*}
This implies the assertion.
\end{proof}

The family of affine vector fields on $\mathbb{R}^{n}$ associated with
(\ref{affine}) is given by%
\begin{equation}
\mathcal{F}=\left\{  X^{u}(x)=A(u)x+Cu+d\left\vert u\in\Omega\right.
\right\}  . \label{F}%
\end{equation}
Then $\mathcal{LA}(\mathcal{F})=\mathcal{LA}(f_{0},f_{1},\ldots,f_{m})$, cf.
(\ref{LARC}). The system group $\mathcal{G}=\mathcal{G}(\mathcal{F})$ is a
subgroup of the semidirect product $\mathbb{R}^{n}\ltimes GL(\mathbb{R}^{n})$.

Since we assume the accessibility rank condition (\ref{LARC}), Jurdjevic
\cite[Theorem 3 on p. 44]{Jurd97} implies that the set $\mathcal{F}$ of vector
fields is transitive on $\mathbb{R}^{n}$, i.e., for all $x,y\in\mathbb{R}^{n}$
there is $g\in\mathcal{G}$ with $y=gx$. Denote by $\mathcal{S}_{\tau
}=\mathcal{S}_{\tau}(\mathcal{F})$ the set of those elements of $\mathcal{S}%
(\mathcal{F})$ with $t_{1}+\cdots+t_{k}=\tau$, and let $\mathcal{S}_{\leq
T}\mathcal{=}\bigcup_{\tau\in\lbrack0,T]}\mathcal{S}_{\tau}$, analogously for
$\mathcal{S}_{\tau}(\mathcal{F}^{R})$. The trajectories for $u\in
\mathcal{U}_{pc}$ of control system (\ref{affine}) are given by the action of
the semigroup $\mathcal{S}$ on $\mathbb{R}^{n}$: For $g\in\mathcal{S}_{\tau}$
and $\tau=t_{k}+\cdots+t_{1},t_{i}>0$,%
\[
gx=g(u)x=\exp(t_{k}X^{u^{k}})\cdots\exp(t_{1}X^{u^{1}})x=\varphi(\tau,x,u),
\]
where $u^{j}\in\Omega,X^{u^{j}}\in\mathcal{F}$, and $\varphi(t,x,u),t\in
\lbrack0,\tau]$, is the solution of (\ref{affine}) with piecewise constant
control $u$ defined, with $t_{0}=0$, by%
\begin{equation}
u(t):=u^{j+1}\text{ for }t\in\left[  \sum\nolimits_{i=0}^{j}t_{i}%
,\sum\nolimits_{i=0}^{j+1}t_{i}\right)  ,\quad j=0,\ldots,k-1. \label{u}%
\end{equation}
Note that $u$ is not uniquely determined by $g$. We will always consider the
interior of $\mathcal{S}$ in the system group $\mathcal{G}$.

\begin{theorem}
\label{Theorem_generalLie}(i) The system semigroup $\mathcal{S}=\mathcal{S}%
(\mathcal{F})$ of (\ref{affine}) satisfies $\mathcal{S}_{\leq\tau}%
\subset\overline{\mathrm{int}(\mathcal{S}_{\leq\tau})}$ in $\mathcal{G}$ for
every $\tau>0$.

(ii) If $g\in\mathrm{int}(\mathcal{S}_{\leq\tau})$ for a $\tau>0$ then
$gx\in\mathrm{int}(\mathcal{O}_{\leq\tau}^{+}(x))$ for every $x\in
\mathbb{R}^{n}$.
\end{theorem}

\begin{proof}
(i) The right invariant vector fields in $\mathcal{F}^{R}$ on $\mathcal{G}$
are real analytic. Since this implies that they are Lie-determined we can
apply Jurdjevic \cite[Corollary on p. 67]{Jurd97} which shows that for every
open set $U$ in an orbit of $\mathcal{F}_{r}$, any $y\in U$, and any $\tau>0$,
the reachable set $\mathcal{S}_{\leq\tau}(\mathcal{F}^{R})(y)\cap U$ contains
an open set in the orbit topology. In particular, this applies to the
reachable set up to time $\tau$ of the identity which by Proposition
\ref{Proposition_semigroups} coincides with $\mathcal{S}_{\leq\tau}$.
Furthermore, \cite[Corollary 1 on p. 68]{Jurd97} implies $\mathcal{S}%
_{\leq\tau}\subset\overline{\mathrm{int}(\mathcal{S}_{\leq\tau})}$.

(ii) The maps $\mathcal{G}\rightarrow\mathbb{R}^{n}:g\mapsto gx$ are open,
hence $g\in\mathrm{int}\left(  \mathcal{S}_{\leq\tau}\right)  $ implies
$gx\in\mathrm{int}(\mathcal{O}_{\leq\tau}^{+}(x))$.
\end{proof}

Next we relate the system semigroup and fixed points in control sets.

\begin{proposition}
\label{Proposition2.5}(i) Let $D\subset\mathbb{R}^{n}$ be a control set with
nonvoid interior for (\ref{affine}). Then for every $x\in\mathrm{int}(D)$
there are $\tau>0$ and $g\in\mathcal{S}_{\tau}\cap\mathrm{int}(\mathcal{S}%
_{\leq\tau+1})$ such that $gx=x$.

(ii) Conversely, let $g\in\mathrm{int}(\mathcal{S})$ with $gx=x$ for some
point $x\in\mathbb{R}^{n}$. Then $x\in\mathrm{int}(D)$ for some control set
$D\subset\mathbb{R}^{n}$.
\end{proposition}

\begin{proof}
(i) Let $x\in\mathrm{int}(D)$. By continuity of the action the set
$H=\{h\in\mathcal{G}\left\vert hx\in\mathrm{int}(D)\right.  \}$ is open in
$\mathcal{G}$, hence for all $t>0$ small enough there exists an element
$h\in\mathcal{S}_{t}\cap H$. By Proposition \ref{Proposition_semigroups}(i) it
follows that for $\sigma_{k}\rightarrow0^{+}$ there are elements $g_{k}%
\in\mathrm{int}(\mathcal{S}_{\leq\sigma_{k}})$ converging to the identity in
$\mathcal{G}$, and hence $h_{k}:=g_{k}h\in\mathrm{int}\left(  \mathcal{S}%
_{\leq t+\sigma_{k}}\right)  \rightarrow h$. Since $H$ is open there is
$k\in\mathbb{N}$ large enough such that $h_{k}\in H\cap\mathcal{S}_{t+\sigma
}\cap\mathrm{int}(\mathcal{S}_{\leq t+\sigma_{k}})$ for some $\sigma\in
\lbrack0,\sigma_{k}]$, hence $h_{k}x\in\mathrm{int}(D)$. By Remark
\ref{Remark2.2}, exact controllability in $\mathrm{int}(D)$ holds. Thus we
find $h_{0}\in\mathcal{S}_{s},s>0$, such that $h_{0}h_{k}x=x$. It follows that
$g:=h_{0}h_{k}\in\mathcal{S}_{\tau}\cap\mathrm{int}(\mathcal{S}_{\leq\tau+1})$
with $\tau:=t+\sigma+s$ and $gx=x$.

(ii) Every $g\in\mathrm{int}(\mathcal{S})$ satisfies $gx\in\mathrm{int}%
(\mathcal{O}^{+}(x))$ and $x\in\mathrm{int}(\mathcal{O}^{-}(gx))$ for all $x$.
Now $gx=x$ implies that $x\in\mathrm{int}(\mathcal{O}^{+}(x))\cap
\mathrm{int}(\mathcal{O}^{-}(x))$, and hence Lemma \ref{Lemma_intersection}
shows that $D=\mathcal{O}^{-}(x)\cap\overline{\mathcal{O}^{+}(x)}$ is a
control set with $x\in\mathrm{int}(D)$.
\end{proof}

Note that $g\in\mathrm{int}(\mathcal{S})$ implies that $g\in\mathrm{int}%
(\mathcal{S}_{\leq\tau})$ for some $\tau>0$; cf. Colonius and Kliemann
\cite[Lemma 4.5.2]{ColK00}. The control $u$ in (\ref{u}) also determines the
element $\Phi_{u}(\tau,0)$ of the system semigroup of the homogeneous bilinear
control system (\ref{hom}). We denote the corresponding semigroup by
$_{\mathbb{R}}\mathcal{S}^{\hom}\subset GL(n,\mathbb{R})$. Theorem
\ref{Theorem_generalLie} and Proposition \ref{Proposition2.5} are also valid
for system (\ref{hom}) and $_{\mathbb{R}}\mathcal{S}^{\hom}$ provided that the
corresponding accessibility rank condition in $\mathbb{R}^{n}\setminus\{0\}$ holds.

Next we describe the relation between the action of the system semigroup and
periodic control functions. We may extend the control defined by (\ref{u}) to
a $\tau$-periodic control function in $\mathcal{U}_{pc}$. The next lemma
follows immediately from Proposition \ref{Proposition_periodicODE}(i).

\begin{lemma}
\label{Lemma_periodic}Let $g(u)\in\mathcal{S}_{\tau}$. Then $g(u)x=x$ for some
$x\in\mathbb{R}^{n}$ if and only if the corresponding $\tau$-periodic
differential equation in (\ref{affine}) has a $\tau$-periodic solution with
initial value $x(0)=x$. This solution is unique if and only if $1\not \in
\mathrm{spec}(\Phi_{u}(\tau,0))$.
\end{lemma}

Note the following results on continuous dependence.

\begin{lemma}
\label{Lemma_path_a1}Let $u,v\in\mathcal{U}_{pc}$ be $\sigma$-periodic and
$\tau$-periodic, resp., for some $\sigma,\tau>0$. Define for $\alpha\in
\lbrack\sigma,\sigma+\tau]$
\[
u^{\alpha}(t):=u(t)\text{ for }t\in\lbrack0,\sigma],\,u^{\alpha}%
(t):=v(t-\sigma)\text{ for }t\in\lbrack\sigma,\alpha],
\]
and extend $u^{\alpha}$ to an $\alpha$-periodic control $u^{\alpha}%
\in\mathcal{U}_{pc}$. Then the controls $u_{\left\vert [0,\sigma+\tau]\right.
}^{\alpha}$ depend continuously on $\alpha$ as elements of $L^{1}%
([0,\sigma+\tau],\mathbb{R}^{m})$. Furthermore, the principal fundamental
solutions $\Phi_{u^{\alpha}}(\alpha,0)\in\mathbb{R}^{n\times n}$ and the
elements $g(u^{\alpha})\in\mathbb{R}^{n}\ltimes GL(\mathbb{R}^{n})$ depend, in
a continuous and piecewise analytic way, on $\alpha$.
\end{lemma}

\begin{proof}
We may write the control $v$ as $v(t)=v^{j}$ for $t\in\lbrack t_{0}%
+\cdots+t_{j-1},t_{0}+\cdots+t_{j}),$ where $v^{j}\in\Omega,t_{0}:=0$ and
$t_{j}>0$ for $j=1,\ldots,\ell$ with $\tau=t_{1}+\cdots+t_{\ell}$. The
assertions follow from the explicit expressions for $\alpha\in\lbrack
\sigma+\sum_{i=0}^{j-1}t_{i},\sigma+\sum_{i=0}^{j}t_{i})$ and $j=1,\ldots
,\ell$,%
\begin{align*}
\Phi_{u^{\alpha}}(\alpha,0) &  =e^{(\alpha-\sigma-\sum_{i=0}^{j-1}%
t_{i})A(v^{j})}e^{t_{i-1}A(v^{j-1})}\cdots e^{t_{1}A(v^{1})}\Phi_{u}%
(\sigma,0),\\
g(u^{\alpha}) &  =\exp\left(  (\alpha-\sigma-\sum\nolimits_{i=0}^{j-1}%
t_{i})X^{v^{j}}\right)  \exp\left(  t_{j-1}X^{v^{j-1}}\right)  \cdots
\exp\left(  t_{1}X^{v^{1}}\right)  g(u),
\end{align*}
where $X^{v^{j}}$ is the affine vector field $X^{v^{j}}(x)=A(v^{j})x+Cv^{j}+d$.
\end{proof}

This result can be used in order to analyze the system semigroup $\mathcal{S}$
of the affine system and the system semigroup $_{\mathbb{R}}\mathcal{S}^{\hom
}$ of the homogeneous part.

\begin{lemma}
\label{Lemma_path_a2}Let $g(u)\in\mathcal{S}_{\sigma}\cap\mathrm{int}%
(\mathcal{S})$ and $g(v)\in\mathcal{S}_{\tau}\cap\mathrm{int}(\mathcal{S})$
for some $\sigma,\tau>0$. Then there exist for $\alpha\in\lbrack0,1]$
$\tau_{\alpha}$-periodic controls $u^{\alpha}$ with $u^{0}=u,\tau_{0}=\sigma$
and $u^{1}=v,\tau_{1}=\tau$ such that the maps $p:[0,1]\rightarrow
\mathrm{int}(\mathcal{S})$ and $p^{\hom}:[0,1]\rightarrow\,_{\mathbb{R}%
}\mathcal{S}^{\hom}$,%
\[
p(\alpha):=g(u^{\alpha})\in\mathrm{int}(\mathcal{S})\text{ and }p^{\hom
}(\alpha):=\Phi_{u^{\alpha}}(\tau_{\alpha},0)\in\,_{\mathbb{R}}\mathcal{S}%
^{\hom},\,\alpha\in\lbrack0,1],
\]
are continuous paths. Furthermore the continuity and smoothness properties
from Lemma \ref{Lemma_path_a1} hold for $u^{a},\Phi_{u^{\alpha}}(\tau_{\alpha
},0)\in\,_{\mathbb{R}}\mathcal{S}^{\hom}$ and $g(u^{\alpha})\in\mathcal{S}$
and also for $\tau_{\alpha}$.
\end{lemma}

\begin{proof}
By Lemma \ref{Lemma_path_a1} one finds a path in $\mathcal{S}$ from $g(u)$ to
$g(v)g(u)$ and an analogous construction yields a path in $\mathcal{S}$ from
$g(v)$ to $g(v)g(u)$. Since for any $g^{\prime}\in\mathcal{S}$ and
$g^{\prime\prime}\in\mathrm{int}(\mathcal{S})$ it follows that $g^{\prime
}g^{\prime\prime},g^{\prime\prime}g^{\prime}\in\mathrm{int}(\mathcal{S})$, the
elements on the paths are in $\mathrm{int}(\mathcal{S})$. Combining these
paths one obtains a path from $g(u)$ to $g(v)$. This can be reparametrized to
obtain a path with $\alpha\in\lbrack0,1]$ and periods $\tau_{\alpha}$ in
$[0,\sigma+\tau]$. Analogously one obtains the path $p^{\hom}$. The smoothness
properties remain valid.
\end{proof}

Next we use these lemmas to prove spectral properties of elements in the
interior of the system semigroup and corresponding periodic solutions.

\begin{proposition}
\label{Proposition_analytic_new}Let $g(u)\in\mathcal{S}_{\sigma}%
\cap\mathrm{int}(\mathcal{S})$ for some $\sigma>0$ with $1\not \in
\mathrm{spec}(\Phi_{u}(\sigma,0))$ and $g(v)\in\mathcal{S}_{\tau}%
\cap\mathrm{int}(\mathcal{S})$ for some $\tau>0$. Consider the paths $p$ in
$\mathrm{int}(\mathcal{S})$ and $p^{\hom}$ in $_{\mathbb{R}}\mathcal{S}^{\hom
}$ constructed in Lemma \ref{Lemma_path_a2}.

(i) For every $\varepsilon>0$ there are $\tau_{\alpha}$-periodic controls
$w^{\alpha}\in\mathcal{U}_{pc}$ and continuous paths $p_{1}:[0,1]\rightarrow
\mathrm{int}(\mathcal{S})$ with $p_{1}(\alpha)=g(w^{\alpha})$ for $\alpha
\in\lbrack0,1]$ and $p_{1}^{\hom}:[0,1]\rightarrow\,_{\mathbb{R}}%
\mathcal{S}^{\hom}$ with $p_{1}^{\hom}(\alpha)=\Phi_{w^{\alpha}}(\tau_{\alpha
},0)$ with%
\begin{align*}
\left\Vert p_{1}(1)-g(v)\right\Vert  & =\left\Vert g(w^{1})-g(v)\right\Vert
<\varepsilon,\\
\left\Vert p_{1}^{\hom}(1)-\Phi_{v}(\tau,0)\right\Vert  & =\left\Vert
\Phi_{w^{1}}(\tau_{1},0)-\Phi_{v}(\tau,0)\right\Vert <\varepsilon,\,
\end{align*}
such that $1\not \in \mathrm{spec}(\Phi_{w^{\alpha}}(\tau_{\alpha},0)$ for all
but at most finitely many $\alpha\in\lbrack0,1]$, and the continuity and
smoothness properties from Lemma \ref{Lemma_path_a1} hold for $u^{\alpha}%
,\tau_{\alpha},\Phi_{w^{\alpha}}(\alpha,0)\in\,_{\mathbb{R}}\mathcal{S}^{\hom
}$, and $g(w^{\alpha})\in\mathcal{S}$.

(ii) If $1\not \in \mathrm{spec}(\Phi_{u^{\alpha_{0}}}(\tau_{\alpha_{0}},0))$
(or $1\not \in \mathrm{spec}(\Phi_{w^{\alpha_{0}}}(\tau_{\alpha_{0}},0))$) for
some $\alpha_{0}\in\lbrack0,1]$, then for all $\alpha$ in a neighborhood of
$\alpha_{0}$ there are unique $\tau^{\alpha}$-periodic solutions with initial
values $x^{\alpha}$ depending continuously on $\alpha$.
\end{proposition}

\begin{proof}
(i) First we prove these properties for the $\alpha$-periodic controls
$u^{\alpha}$ constructed in Lemma \ref{Lemma_path_a1}. The proof will proceed
inductively for $j=0,\ldots,\ell-1$ and $t\in\lbrack t_{j},t_{j+1}]$. Since
$1\not \in \mathrm{spec}(\Phi_{u}(\sigma,0)$ it follows that the analytic
function $\alpha\mapsto\det(I-\Phi_{u^{\alpha}}(\tau^{\alpha},0))$ is not
identically $0$. Hence there are at most finitely many $\alpha_{i}\in
\lbrack0,t_{1}]$ with $1\in\mathrm{spec}(\Phi_{u^{\alpha_{i}}}(\tau
^{\alpha_{i}},0)$. If $\det(I-\Phi_{u^{t_{1}}}(\tau_{t_{1}},0))\not =0$,
choose $s_{1}=t_{1}$, otherwise choose $s_{1}<t_{1}$ arbitrarily close to
$t_{1}$ with $\det(I-\Phi_{u^{s_{1}}}(\tau_{s_{1}},0))\not =0$, and define for
$\alpha\in\lbrack\sigma,\sigma+s_{1}]$%
\[
w^{\alpha}(t)=u^{\alpha}(t),\,t\in\lbrack0,\alpha].
\]
Then the $\alpha$-periodic extension of $w^{\alpha}$ satisfies $p_{1}%
(\alpha):=g(w^{\alpha})\in\mathrm{int}(\mathcal{S})$ and $p_{1}^{\hom}%
(\alpha):=\Phi_{w^{\alpha}}(\alpha,0)\in\,_{\mathbb{R}}\mathcal{S}^{\hom}$
with $1\not \in \mathrm{spec}(\Phi_{w^{\alpha}}(\alpha,0))$ for all but at
most finitely many $\alpha\in\lbrack\sigma,\sigma+s_{1}]$. Consider%
\[
I-e^{(\alpha-\sigma-s_{1})A(v^{2})}e^{s_{1}A(v^{1})}\Phi_{u}(\sigma,0)\text{
for }\alpha\in\lbrack\sigma+s_{1},\sigma+t_{2}].
\]
For $\alpha=\sigma+s_{1}$ the determinant of this matrix is unequal zero,
hence it has at most finitely many zeros in $[s_{1},t_{2}]$.

Proceeding in this way up to $j=\ell-1$ one constructs $\alpha$-periodic
controls $w^{\alpha}$ such that $p_{1}(\alpha):=g(w^{\alpha})\in
\mathrm{int}(\mathcal{S})$ and $p_{1}^{\hom}(\alpha):=\Phi_{w^{\alpha}}%
(\alpha,0)\in\,_{\mathbb{R}}\mathcal{S}^{\hom}$ with $1\not \in
\mathrm{spec}(\Phi_{w^{\alpha}}(\alpha,0))$ for all but at most finitely many
$\alpha\in\lbrack\sigma,\sigma+\tau]$. Since $t_{i}-s_{i}>0$ is arbitrarily
small it also follows that
\[
\left\Vert p_{1}^{\hom}(\sigma+\tau)-p^{\hom}(\sigma+\tau)\right\Vert
<\varepsilon,\,\left\Vert p_{1}(\sigma+\tau)-p(\sigma+\tau)\right\Vert
<\varepsilon.
\]
The same constructions as in the proof of Lemma \ref{Lemma_path_a2} can also
be applied here and yield assertion (i).

(ii) Continuous dependence on $\alpha$ of $\Phi_{w^{\alpha}}(\tau_{\alpha},0)$
implies that $1\not \in \mathrm{spec}(\Phi_{w^{\alpha}}(\tau_{\alpha},0)$ for
all $\alpha$ in a neighborhood of $\alpha_{0}$, hence by Proposition
\ref{Proposition_periodicODE}(i) there are unique $\tau_{\alpha}$-periodic
solutions. Since the controls $w^{\alpha}$ depend continuously on $\alpha$ as
elements of $L^{1}([0,1],\mathbb{R}^{m}\dot{)}$ also
\[
A(w^{\alpha}(\cdot))=A+\sum_{i=1}^{m}w_{i}^{\alpha}(\cdot)B_{i}\in
L^{1}(\left[  0,\sigma+\tau\right]  ;\mathbb{R}^{n\times n}),\,Cw^{\alpha
}(\cdot)+d\in L^{1}(\left[  0,\sigma+\tau\right]  ;\mathbb{R}^{n})
\]
depend continuously on $\alpha$. Thus Proposition
\ref{Proposition_periodicODE}(iv) shows that their initial values $x^{\alpha}$
depend continuously on $\alpha$.
\end{proof}

The next two lemmas discuss the periodic solutions when $1$ is in the spectrum.

\begin{lemma}
\label{Lemma_path3}Consider, for $k=0,1,\ldots$, the differential equations%
\[
\dot{x}(t)=A(u^{k}(t))x(t)+Cu^{k}(t)+d,
\]
where $u^{k}$ is $\tau_{k}$-periodic with $\tau_{k}\rightarrow\tau_{0}>0$ and
$u^{k}\rightarrow u^{0}$ in $L^{1}(\left[  0,\tau_{0}+1\right]  ;\mathbb{R}%
^{m})$ with $\left\Vert u_{k}\right\Vert _{\infty}\leq c,k\in\mathbb{N}$, for
some $c>0$, and

(i) the principal fundamental solutions $\Phi_{u^{k}}(t,s)$ of $\dot
{x}(t)=A(u^{k}(t))x(t)$ satisfy $1\in\mathrm{spec}(\Phi_{u^{0}}(\tau_{0},0))$
and $1\not \in \mathrm{spec}(\Phi_{u^{k}}(\tau_{k},0))$ for $k=1,2,\ldots$,

(ii) $\int_{0}^{\tau_{0}}\Phi_{u^{0}}(\tau_{0},s)\left(  Cu^{0}(s)+d\right)
ds\not \in \operatorname{Im}(I-\Phi_{u^{0}}(\tau_{0},0))$.

Then there are unique $\tau_{k}$-periodic solutions for $u^{k}$ and their
initial values $x^{k}$ satisfy%
\begin{equation}
\left\Vert x^{k}\right\Vert \rightarrow\infty\text{ and }\frac{x^{k}%
}{\left\Vert x^{k}\right\Vert }\rightarrow\ker(I-\Phi_{u^{0}}(\tau
_{0},0))=\mathbf{E}(\Phi_{u^{0}}(\tau_{0},0);1)\text{ for }k\rightarrow\infty.
\label{CRS_5.8}%
\end{equation}

\end{lemma}

\begin{proof}
By Proposition \ref{Proposition_periodicODE}(i) it follows that for every
$k=1,2,\ldots$ there is a unique $\tau_{k}$-periodic solution with initial
value $x^{k}$ satisfying%
\begin{equation}
(I-\Phi_{u^{k}}(\tau_{k},0))x^{k}=\int_{0}^{\tau_{k}}\Phi_{u^{k}}(\tau
_{k},s)\left(  Cu^{k}(s)+d\right)  ds.\label{4.4}%
\end{equation}
If $x^{k}$ remains bounded we may suppose that there is $x^{0}\in
\mathbb{R}^{n}$ with $x^{k}\rightarrow x^{0}$. Since $u^{k}\rightarrow u^{0}$
in $L^{1}(\left[  0,\tau_{0}+1\right]  ;\mathbb{R}^{m})$ it follows that also%
\[
A(u^{k}(\cdot))\rightarrow A(u^{0}(\cdot))\text{ and }Cu^{k}(\cdot
)+d\rightarrow Cu^{0}(\cdot)+d
\]
in $L^{1}(\left[  0,\tau_{0}+1\right]  ;\mathbb{R}^{n\times n})$ and in
$L^{1}(\left[  0,\tau_{0}+1\right]  ;\mathbb{R}^{n})$, resp. Thus, by
Proposition \ref{Proposition_periodicODE}(iii), the right hand sides of
(\ref{4.4}) converge to $\int_{0}^{\tau_{0}}\Phi_{u^{0}}(\tau_{0},s)\left(
Cu^{0}(s)+d\right)  ds$, and one obtains a contradiction to assumption (ii).
This shows that $\left\Vert x^{k}\right\Vert \rightarrow\infty$. Similarly
also the second assertion in (\ref{CRS_5.8}) follows when we divide
(\ref{4.4}) by $\left\Vert x^{k}\right\Vert $.
\end{proof}

The next lemma describes the case where assumption (ii) above is violated.

\begin{lemma}
\label{Lemma_image}Consider for a $\tau_{0}$-periodic control $u^{0}$%
\begin{equation}
\dot{x}(t)=A(u^{0}(t))x(t)+Cu^{0}(t)+d, \label{affine_0}%
\end{equation}
and suppose that the principal fundamental solution of $\dot{x}(t)=A(u^{0}%
(t))x(t)$ satisfies $1\in\mathrm{spec}(\Phi_{u^{0}}(\tau_{0},0))$ and
$\int_{0}^{\tau_{0}}\Phi_{u^{0}}(\tau_{0},s)\left(  Cu^{0}(s)+d\right)
ds=(I-\Phi_{u^{0}}(\tau_{0},0))y^{0}$ for some $y^{0}\in\mathbb{R}^{n}$. Then
the nontrivial affine subspace $Y:=y^{0}+\mathbf{E}(\Phi_{u^{0}}(\tau
_{0},0);1))$ has the property that there is a\ $\tau_{0}$-periodic solution of
(\ref{affine_0}) starting in $y$ if and only if $y\in Y$, and there are
$x^{k}\in Y,k\in\mathbb{N}$, satisfying the conditions in (\ref{CRS_5.8}).
\end{lemma}

\begin{proof}
The first assertion is clear by the definitions. The second assertion follows
by choosing $x^{k}:=y^{0}+kz,k\in\mathbb{N}$, where $0\not =z\in
\mathbf{E}(\Phi_{u^{0}}(\tau_{0},0);1)$.
\end{proof}

\section{Control sets for hyperbolic systems\label{Section5}}

In this section we present definitions of hyperbolicity for affine control
systems and show that hyperbolic systems have a unique control set with
nonvoid interior and that it is bounded.

Since for any $\tau$-periodic control, the homogeneous part (\ref{hom_n}) of
affine system (\ref{affine}) is a homogeneous periodic differential equation,
we can define corresponding Floquet multipliers which are the eigenvalues of
the principal fundamental solution $\Phi_{u}(\tau,0)$; cf. Subsection
\ref{Subsection2.3}.

\begin{definition}
\label{Definition_hyperbolic}An affine system of the form (\ref{affine}) is
hyperbolic if%
\begin{equation}
1\not \in \mathrm{spec}(\Phi_{u}(\tau,0))\text{ for all }\tau\text{-periodic
}u\in\mathcal{U}_{pc}\text{ with }\tau>0\text{, }g(u)\in\mathcal{S}_{\tau}%
\cap\mathrm{int}(\mathcal{S}). \label{hyp}%
\end{equation}
Otherwise it is called nonhyperbolic.
\end{definition}

\begin{remark}
If there is a $\tau$-periodic control $u$ with $g(u)\in\mathcal{S}_{\tau}%
\cap\mathrm{int}(\mathcal{S})$ and $\rho\in\mathrm{spec}(\Phi_{u}(\tau,0))$
with $\rho^{k}=1$ for some $k\in\mathbb{N}$, then the system is nonhyperbolic.
In fact, we may consider $u$ as a $k\tau$-periodic control and find that
$1\in\mathrm{spec}(\Phi_{u}(k\tau,0))$ with $\hat{g}(u)\in\mathcal{S}_{k\tau
}\cap\mathrm{int}(\mathcal{S})$.
\end{remark}

\begin{remark}
If our sufficient condition for the existence of a control set $_{\mathbb{R}%
}D^{\hom}$ of the homogeneous part of (\ref{affine}) holds (cf. Theorem
\ref{Theorem_95}(iii)) there is a Floquet exponent $0=\frac{1}{\tau}%
\log\left\vert \rho\right\vert $ for a Floquet multiplier $\rho\in
\mathrm{spec}(\Phi_{u}(\tau,0))$, hence $\left\vert \rho\right\vert =1$.
According to the preceding remark, the system can only be hyperbolic, if
$\rho$ is not a root of unity.
\end{remark}

Next we show that, for hyperbolic affine systems, there is a unique control
set with nonvoid interior.

\begin{theorem}
\label{Theorem_hyperbolic1}Suppose that the affine system (\ref{affine}) is
hyperbolic. Then there is a unique control set $D$ with nonvoid interior, and
for every $g\in\mathrm{int}(\mathcal{S})$ there is a unique $x\in
\mathbb{R}^{n}$ with $x=gx$ and%
\[
\mathrm{int}(D)=\left\{  x\in\mathbb{R}^{n}\left\vert \text{there is }%
g\in\mathrm{int}(\mathcal{S})\text{ with }x=gx\right.  \right\}  .
\]

\end{theorem}

\begin{proof}
Let $g=g(u)\in\mathcal{S}_{\tau}\cap\mathrm{int}(\mathcal{S})$. By
hyperbolicity, $1$ is not an eigenvalue of the principal fundamental solution
$\Phi_{u}(\tau,0)$. Proposition \ref{Proposition_periodicODE}(i) implies that
there is a unique $\tau$-periodic solution starting in some $x\in
\mathbb{R}^{n}$, hence $gx=x$ by Lemma \ref{Lemma_periodic}. By Proposition
\ref{Proposition2.5}(ii) it follows that $x\in\mathrm{int}(D)$ for some
control set $D$. In order to show that $D$ does not depend on $g$, consider
$g,h\in\mathrm{int}(\mathcal{S})$. By Lemma \ref{Lemma_path_a2}, one finds a
continuous path $p$ in $\mathrm{int}(\mathcal{S})$ from $g$ to $h$. For all
$\alpha\in\lbrack0,1]$ hyperbolicity implies that $1\not \in \mathrm{spec}%
(\Phi_{u^{\alpha}}(\tau_{\alpha},0))$ and hence there are unique fixed point
$x^{\alpha}$ of $g(u^{\alpha})$ and a control set $D^{\alpha}$ with
$x^{\alpha}\in\mathrm{int}(D^{\alpha})$. As in Proposition
\ref{Proposition_analytic_new}(ii) it follows that also $x^{\alpha}$ depends
continuously on $\alpha$. Hence for small $\alpha>0$ all points $x^{\alpha}$
are contained in a single control set $D$ showing%
\[
\alpha^{\ast}:=\sup\left\{  \alpha\left\vert x^{\alpha^{\prime}}\in D\text{
for all }\alpha^{\prime}\in\lbrack0,\alpha]\right.  \right\}  >0.
\]
Since $x^{\alpha^{\ast}}\in\mathrm{int}(D^{\alpha^{\ast}})$ it follows from
Remark \ref{Remark2.2} that $D^{\alpha^{\ast}}=D$ which shows that
$x^{\alpha^{\ast}}\in\mathrm{int}(D)$. If $\alpha^{\ast}<1$ this implies that
$D^{\alpha}=D$ for all $\alpha\in\lbrack\alpha^{\ast},\alpha^{\ast
}+\varepsilon]$ for some $\varepsilon>0$ contradicting the definition of
$\alpha^{\ast}$. It follows that $\alpha^{\ast}=1$ and hence there is a single
control set $D$ containing all $x=gx$ for $g\in\mathrm{int}(\mathcal{S})$. The
corresponding periodic controls generate periodic solution which also are
contained in $D$.

It remains to show that~$D$ is the unique control set with nonvoid interior.
By Proposition \ref{Proposition2.5}(i), for a point $x$ in the interior of any
control set, there are $\tau>0$ and $g\in\mathcal{S}_{\tau}\cap\mathrm{int}%
\left(  \mathcal{S}_{\leq\tau+1}\right)  $ with $gx=x$. Hence it follows that
$x\in D$.
\end{proof}

The question arises if the control set $D$ is bounded. We will give a positive
answer provided that the following uniform hyperbolicity condition holds
assuming that the control range $\Omega$ is a compact and convex neighborhood
of the origin in $\mathbb{R}^{m}$ and hence, for system (\ref{hom}), the
control flow $\mathbf{\Psi}$ on $\mathcal{U}\times\mathbb{R}^{n}$ is well
defined (cf. Subsection \ref{Subsection2.1}).

\begin{definition}
\label{Definition_uni_hyperbolic}The homogeneous bilinear system (\ref{hom})
is uniformly hyperbolic if the vector bundle $\mathcal{U}\times\mathbb{R}^{n}$
can be decomposed into the Whitney sum of two invariant subbundles
$\mathcal{V}^{1}$ and $\mathcal{V}^{2}$ such that the restrictions
$\mathbf{\Psi}^{1}$ and $\mathbf{\Psi}^{2}$ of the control flow $\mathbf{\Psi
}$ to $\mathcal{V}^{1}$ and $\mathcal{V}^{2}$, resp., satisfy for constants
$\alpha>0$ and $K\geq1$ and for all $(u,x_{i})\in\mathcal{V}^{i}$%
\begin{align*}
\left\Vert \varphi(t,x_{1},u)\right\Vert  &  =\left\Vert \mathbf{\Psi}_{t}%
^{1}(u,x_{1})\right\Vert \leq Ke^{-\alpha t}\left\Vert x_{1}\right\Vert \text{
for }t\geq0\text{,}\\
\left\Vert \varphi(t,x_{2},u)\right\Vert  &  =\left\Vert \mathbf{\Psi}_{t}%
^{2}(u,x_{2})\right\Vert \leq Ke^{\alpha t}\left\Vert x_{2}\right\Vert \text{
for }t\leq0.
\end{align*}

\end{definition}

\begin{remark}
The uniform hyperbolicity condition is also used in Kawan \cite{Kawa16}, Da
Silva and Kawan \cite{DaSK16}. It is equivalent to the condition that $0$ is
not in the Sacker-Sell spectrum of the linear flow $\mathbf{\Psi}$; cf.
Colonius and Kliemann \cite[Section 5.5]{ColK00}.
\end{remark}

Then, for $i=1,2\,$, one obtains that $\mathcal{V}^{i}(u):=\{x\in
\mathbb{R}^{n}\left\vert (u,x)\in\mathcal{V}^{i}\right.  \}$ is a subspace of
$\mathbb{R}^{n}$ and its dimension is independent of $u\in\mathcal{U}$. For
all $u\in\mathcal{U}$%
\[
\mathbb{R}^{n}=\mathcal{V}^{1}(u)\oplus\mathcal{V}^{2}(u)\text{ and }%
\varphi(t,x_{i},u)\in\mathcal{V}^{i}(u(t+\cdot))\text{ for all }t\in
\mathbb{R},
\]
hence, for $x=x_{1}\oplus x_{2}$ with $x_{i}\in\mathcal{V}^{i}(u)$ and
$\Phi_{u}^{i}(t,s):=\Phi_{u}(t,s)_{\left\vert \mathcal{V}_{i}(u(s+\cdot
))\right.  }$ for $t,s\in\mathbb{R}$,
\[
\varphi(t,x,u)=\varphi(t,x_{1},u)\oplus\varphi(t,x_{2},u)\text{ and }\Phi
_{u}(t,s)=\Phi_{u}^{1}(t,s)+\Phi_{u}^{2}(t,s).
\]
The uniform hyperbolicity condition above implies that system (\ref{affine})
is hyperbolic in the sense of Definition \ref{Definition_hyperbolic} since%
\[
\mathrm{spec}(\Phi_{u}(\tau,0))=\mathrm{spec}(\Phi_{u}^{1}(\tau,0))\cup
\mathrm{spec}(\Phi_{u}^{2}(\tau,0)),
\]
and $\rho\in\mathrm{spec}(\Phi_{u}^{1}(\tau,0))$ implies $\left\vert
\rho\right\vert \leq e^{-\alpha\tau}$, $\rho\in\mathrm{spec}(\Phi_{u}^{2}%
(\tau,0))$ implies $\left\vert \rho\right\vert \geq e^{\alpha\tau}$.

\begin{lemma}
\label{Lemma_bounded}Suppose that the uniform hyperbolicity assumption holds.
Then there is $c>0$ such that for $(u,x_{1})\in\mathcal{V}^{1}$ and
$(u,x_{2})\in\mathcal{V}^{2}$%
\[
\left\Vert \varphi(t,x_{1},u)\right\Vert \leq K\left\Vert x_{1}\right\Vert
+\frac{Kc}{\alpha}\text{ for }t\geq0,\quad\left\Vert \varphi(t,x_{2}%
,u)\right\Vert \leq K\left\Vert x_{2}\right\Vert +\frac{Kc}{\alpha}\text{ for
}t\leq0.
\]

\end{lemma}

\begin{proof}
Denote the projections of $\mathbb{R}^{n}$ to $\mathcal{V}^{1}(u)$ along
$\mathcal{V}^{2}(u)$ by $P_{u}$ and choose $c>0$ such that $\left\Vert
P_{u}\right\Vert \left\Vert Cv+d\right\Vert \leq c$ for all $u\in
\mathcal{U},v\in\Omega$. By invariance of $\mathcal{V}^{1}$, $P_{u(t+\cdot
)}\Phi_{u}(t,s)=\Phi_{u}(t,s)P_{u(s+\cdot)}$, and hence
\begin{align*}
\varphi(t,x_{1},u)  &  =P_{u(t+\cdot)}\varphi(t,x_{1},u)=P_{u(t+\cdot)}%
\Phi_{u}(t,0)x_{1}+\int_{0}^{t}P_{u(t+\cdot)}\Phi_{u}(t,s)[Cu(s)+d]ds\\
&  =\Phi_{u}^{1}(t,0)x_{1}+\int_{0}^{\tau}\Phi_{u}^{1}(t,s)P_{u(s+\cdot
)}[Cu(s)+d]ds.
\end{align*}
Then it follows for all $u\in\mathcal{U}$ and $t\geq0$ that%
\begin{align*}
\left\Vert \varphi(t,x_{1},u)\right\Vert  &  \leq\left\Vert \Phi_{u}%
^{1}(t,0)x_{1}\right\Vert +\int_{0}^{t}\left\Vert \Phi_{u}^{1}%
(t,s)P_{u(s+\cdot)}[Cu(s)+d]\right\Vert ds\\
&  \leq Ke^{-\alpha t}\left\Vert x_{1}\right\Vert +Kc\int_{0}^{t}%
e^{-\alpha(t-s)}ds\leq K\left\Vert x_{1}\right\Vert +\frac{Kc}{\alpha}.
\end{align*}
The second assertion is shown analogously.
\end{proof}

\begin{theorem}
\label{Theorem_hyperbolic2}Let $\Omega$ be a compact and convex neighborhood
of the origin in $\mathbb{R}^{m}$. Suppose that the homogeneous part
(\ref{hom_n}) of affine system (\ref{affine}) satisfies the uniform
hyperbolicity condition in Definition \ref{Definition_uni_hyperbolic}. Then
the unique control set $D$ with nonvoid interior is bounded.
\end{theorem}

\begin{proof}
We show that $\mathrm{int}(D)$ is bounded. This will yield the assertion,
since by Remark \ref{Remark2.2} the accessibility rank condition implies that
$D\subset\overline{\mathrm{int}(D)}$. Fix $x\in\mathrm{int}(D)$ and consider
an arbitrary point $y\in\mathrm{int}(D)=\mathcal{O}^{+}(x)\cap\mathcal{O}%
^{-}(x)$. Thus there are controls $u^{1},u^{2}\in\mathcal{U}$ and times
$t_{1},t_{2}>0$ with $y=\varphi(t_{1},x,u^{1})=\varphi(-t_{2},x,u^{2})$.
Define
\[
u(t)=\left\{
\begin{array}
[c]{lll}%
u^{1}(t) & \text{for} & t\in\lbrack0,t_{1}]\\
u^{2}(t) & \text{for} & t\in(-t_{2},0)
\end{array}
\right.
\]
and extend $u$ to a $t_{1}+t_{2}$-periodic function on $\mathbb{R}$. Then
$u(t_{1}+\cdot)=u(-t_{2}+\cdot)$ in $\mathcal{U}$ implying%
\begin{equation}
\mathcal{V}^{i}(u(t_{1}+\cdot))=\mathcal{V}^{i}(u(-t_{2}+\cdot))\text{ for
}i=1,2. \label{decompose0}%
\end{equation}
We decompose
\[
x=x_{1}\oplus x_{2}\text{ with }x_{i}\in\mathcal{V}^{i}(u)\text{ and }%
y=y_{1}\oplus y_{2}\text{ with }y_{i}\in\mathcal{V}^{i}(u(t_{1}+\cdot))\text{
for }i=1,2.
\]
The invariance of the complementary subbundles $\mathcal{V}^{i}$ together with
(\ref{decompose0}) shows that%
\[
y_{1}=\varphi(t_{1},x_{1},u)=\varphi(-t_{2},x_{1},u)\text{ and }y_{2}%
=\varphi(t_{1},x_{2},u)=\varphi(-t_{2},x_{2},u),
\]
and $y=\varphi(t_{1},x_{1},u)\oplus\varphi(-t_{2},x_{2},u)$. By Lemma
\ref{Lemma_bounded} $\left\Vert \varphi(t_{1},x_{1},u)\right\Vert $ and
$\left\Vert \varphi(-t_{2},x_{2},u)\right\Vert $ and hence all $y\in
\mathrm{int}(D)$ satisfy bounds which are independent of $t_{1}$ and $t_{2}$.
\end{proof}

Next we present a simple example of a\ uniformly hyperbolic affine system.

\begin{example}
Consider the following system with control range $\Omega=[-1,1]$%
\[
\left(
\begin{array}
[c]{c}%
\dot{x}\\
\dot{y}%
\end{array}
\right)  =\left(
\begin{array}
[c]{cc}%
2 & 0\\
0 & -2
\end{array}
\right)  \left(
\begin{array}
[c]{c}%
x\\
y
\end{array}
\right)  +u(t)\left(
\begin{array}
[c]{cc}%
1 & 0\\
0 & 1
\end{array}
\right)  \left(
\begin{array}
[c]{c}%
x\\
y
\end{array}
\right)  +\left(
\begin{array}
[c]{c}%
3\\
3
\end{array}
\right)  u(t)+\left(
\begin{array}
[c]{c}%
3\\
0
\end{array}
\right)  .
\]
For the homogeneous part $A(u)=A+uB$ the exponential growth rates of the $x$-
and the $y$-component are for $u\in\lbrack-1,1]$ given by $\lambda
_{1}(u)=2+u\geq1$ and $\lambda_{2}(u)=-2+u\leq-1$, resp. Thus the system is
uniformly hyperbolic with $\mathcal{V}^{1}=\mathcal{U}\times(\{0\}\times
\mathbb{R})$ and $\mathcal{V}^{2}=\mathcal{U}\times(\mathbb{R}\times\{0\})$.
We claim that the unique control set\ with nonvoid interior is $D=\left(
-2,0\right)  \times\lbrack-1,3]$. For the proof, first observe that the
equilibria are given by%
\[
0=(2+u)x_{u}+3u+3,~0=(-2+u)y_{u}+3u\text{, hence }x_{u}=-\frac{3u+3}%
{2+u}\text{ and }y_{u}=\frac{3u}{2-u}.
\]
The maps $u\mapsto x_{u}$ and $u\mapsto y_{u}$ are monotonically decreasing
and increasing, resp., since $\frac{d}{du}x_{u}<0$ and $\frac{d}{du}y_{u}>0$.
This implies that the set of equilibria is contained in%
\[
\lbrack x_{1},x_{-1}]\times\lbrack y_{-1},y_{1}]=\left[  -2,0\right]
\times\lbrack-1,3].
\]
Inspection of the phase portraits for constant $u$ shows that any control set
is contained in this set, and the phase portraits for $u=-1$ and $u=1$ show
that one can approximately reach (with a combination of these controls) from
any point $(x,y)^{\top}$ in $\left(  -2,0\right)  \times\lbrack-1,3]$ any
other point in this set while this is not possible from points $(-2,y)^{\top}$
and $(0,y)^{\top},y\in\lbrack-1,3]$. This proves the claim.
\end{example}

\section{Affine control systems and projective spaces\label{Section6}}

In this section we construct for affine control systems and their homogeneous
parts induced systems on projective spaces. In order to distinguish explicitly
between control sets and chain control sets referring to the affine system and
its homogenous part, we will mark the latter by the suffix \textquotedblleft%
$\hom$\textquotedblright\ in the rest of this paper.

System (\ref{affine}) and (\ref{hom_n}) can be embedded into a homogeneous
bilinear control system in $\mathbb{R}^{n+1}$ of the form (cf. Elliott
\cite[Subsection 3.8.1]{Elliott})%
\begin{equation}
\left(
\begin{array}
[c]{c}%
\dot{x}(t)\\
\dot{z}(t)
\end{array}
\right)  =\left(
\begin{array}
[c]{cc}%
A & d\\
0 & 0
\end{array}
\right)  \left(
\begin{array}
[c]{c}%
x(t)\\
z(t)
\end{array}
\right)  +\sum_{i=1}^{m}u_{i}(t)\left(
\begin{array}
[c]{cc}%
B_{i} & c_{i}\\
0 & 0
\end{array}
\right)  \left(
\begin{array}
[c]{c}%
x(t)\\
z(t)
\end{array}
\right)  . \label{hom_n+1}%
\end{equation}
Denote the solutions of (\ref{hom_n+1}) with initial condition
$(x(0),z(0))=(x^{0},z^{0})\in\mathbb{R}^{n}\times\mathbb{R}$ by $\psi
(t,\left(  x^{0},z^{0}\right)  ,u),t\in\mathbb{R}$. For initial values of the
form $(x^{0},1)\in\mathbb{R}^{n+1}$ one finds%
\begin{equation}
\psi(t,\left(  x^{0},1\right)  ,u)=\left(  \varphi(t,x^{0},u),1\right)  \text{
in }\mathbb{R}^{n+1}, \label{hom1}%
\end{equation}
and for initial values of the form $(x^{0},0)\in\mathbb{R}^{n+1}$ one finds
\begin{equation}
\psi(t,\left(  x^{0},0\right)  ,u)=\left(  \varphi_{\hom}(t,x^{0},u),0\right)
\text{ in }\mathbb{R}^{n+1}. \label{hom0}%
\end{equation}
Thus the trajectories (\ref{hom1}) and (\ref{hom0}) are copies of the
trajectories of (\ref{affine}) and of its homogeneous part (\ref{hom_n}),
resp., obtained by adding a trivial $\left(  n+1\right)  $st. component. An
immediate consequence is the following proposition.

\begin{proposition}
\label{Proposition_control}(i) A subset $D\subset\mathbb{R}^{n}$ is a control
set of (\ref{affine}) if and only if the set $D^{1}:=\{(x,1)\left\vert x\in
D\right.  \}$ is a control set of (\ref{hom_n+1}) in $\mathbb{R}%
^{n+1}\setminus\{0\}$.

(ii) A subset $_{\mathbb{R}}D^{\hom}\subset\mathbb{R}^{n}\setminus\{0\}$ is a
control set of (\ref{hom_n}) if and only if the set $D^{0}:=\{(x,0)\left\vert
x\in\,_{\mathbb{R}}D^{\hom}\right.  \}$ is a control set of (\ref{hom_n+1}) in
$\mathbb{R}^{n+1}\setminus\{0\}$.
\end{proposition}

Next we discuss associated systems in projective spaces. Recall that
$\mathbb{P}^{n-1}=(\mathbb{R}^{n}\setminus\{0\})/\thicksim$, where $\thicksim$
is the equivalence relation $x\thicksim y$ if $y=\lambda x$ with some
$\lambda\not =0$. An atlas of $\mathbb{P}^{n-1}$ is given by $n$ charts
$(U_{i},\psi_{i})$, where $U_{i}$ is the set of equivalence classes
$[x_{1}:\cdots:x_{n}]$ with $x_{i}\not =0$ (using homogeneous coordinates) and
$\psi_{i}:U_{i}\rightarrow\mathbb{R}^{n-1}$ is defined by%
\[
\psi_{i}([x_{1}:\cdots:x_{n}])=\left(  \frac{x_{1}}{x_{i}},\ldots
,\widehat{\frac{x_{i}}{x_{i}}},\ldots,\frac{x_{n}}{x_{i}}\right)  ;
\]
here the hat means that the $i$-th entry is missing. Denote by $\pi
_{\mathbb{P}}$ both projections $\mathbb{R}^{n}\rightarrow\mathbb{P}^{n-1}$
and $\mathbb{R}^{n+1}\rightarrow\mathbb{P}^{n}$. A metric on $\mathbb{P}^{n}$
is given by defining for elements $p_{1}=\pi_{\mathbb{P}}x,p_{2}%
=\pi_{\mathbb{P}}y$%
\begin{equation}
d(p_{1},p_{2})=\min\left\{  \left\Vert \frac{x}{\left\Vert x\right\Vert
}-\frac{y}{\left\Vert y\right\Vert }\right\Vert ,\left\Vert \frac
{x}{\left\Vert x\right\Vert }+\frac{y}{\left\Vert y\right\Vert }\right\Vert
\right\}  . \label{metric_P}%
\end{equation}
Projecting the homogeneous bilinear control system (\ref{hom_n+1}) in
$\mathbb{R}^{n+1}$ to $\mathbb{P}^{n}$ one obtains the following system given
in homogeneous coordinates by%
\begin{equation}
\left[
\begin{array}
[c]{c}%
\dot{x}(t)\\
\dot{z}(t)
\end{array}
\right]  =\left(  \left(
\begin{array}
[c]{cc}%
A & d\\
0 & 0
\end{array}
\right)  +\sum_{i=1}^{m}u_{i}(t)\left(
\begin{array}
[c]{cc}%
B_{i} & c_{i}\\
0 & 0
\end{array}
\right)  \right)  \left[
\begin{array}
[c]{c}%
x(t)\\
z(t)
\end{array}
\right]  . \label{Pn}%
\end{equation}
Projective space $\mathbb{P}^{n}$ can be written as the disjoint union
$\mathbb{P}^{n}=\mathbb{P}^{n,1}\dot{\cup}\mathbb{P}^{n,0}$, where, in
homogeneous coordinates, the levels $\mathbb{P}^{n,i}$ are given by%
\[
\mathbb{P}^{n,i}:=\left\{  [x_{1}:\cdots:x_{n}:i]\left\vert (x_{1}%
,\ldots,x_{n})\in\mathbb{R}^{n}\right.  \right\}  \text{ for }i=0,1.
\]
Observe that, by homogeneity, $\mathbb{P}^{n,0}=\left\{  [x_{1}:\cdots
:x_{n}:0]\left\vert \text{ }\left\Vert (x_{1},\ldots,x_{n})\right\Vert
=1\right.  \right\}  $. Any trajectory of system (\ref{Pn}) is obtained as the
projection of a trajectory of (\ref{hom_n+1}) with initial condition
satisfying $z^{0}=0$ or $1$, since any initial value $[x_{1}^{0}:\cdots
:x_{n}^{0}:z^{0}]$ with $z^{0}\not =0$ coincides with $[\frac{x_{1}^{0}}%
{z^{0}}:\cdots:\frac{x_{n}^{0}}{z^{0}}:1]$.

Loosely speaking, $\mathbb{P}^{n,0}$ is projective space $\mathbb{P}^{n-1}$
(embedded into $\mathbb{P}^{n}$) and $\mathbb{P}^{n,1}$ is $\mathbb{P}^{n}$
without $\mathbb{P}^{n-1}$. In fact, as noted above, an atlas of
$\mathbb{P}^{n}$ is given by $n+1$ charts $(U_{i},\psi_{i})$. A trivial atlas
for $\mathbb{P}^{n,1}$ is given by $\left\{  (U_{n+1},\psi_{n+1})\right\}  $
proving that $\mathbb{P}^{n,1}$ is a manifold which is diffeomorphic to
$\mathbb{R}^{n}$. The space $\mathbb{P}^{n,0}$ is closed in $\mathbb{P}^{n}$,
and the spaces $\mathbb{P}^{n-1}$ and $\mathbb{P}^{n,0}$ are diffeomorphic
under the map%
\begin{equation}
e:\mathbb{P}^{n-1}\rightarrow\mathbb{P}^{n,0}:[x_{1}:\cdots:x_{n}%
]\mapsto\lbrack x_{1}:\cdots:x_{n}:0]. \label{e}%
\end{equation}
For any trajectory $\psi(t,\left(  x^{0},1\right)  ,u)=\left(  \varphi
_{1}(t,x^{0},u),\ldots,\varphi_{n}(t,x^{0},u),1\right)  ^{\top}$ of system
(\ref{hom_n+1}) in $\mathbb{R}^{n+1}\setminus\{0\}$, the projection to
$\mathbb{P}^{n,1}\subset\mathbb{P}^{n}$ is $[\varphi_{1}(t,x^{0}%
,u):\cdots:\varphi_{n}(t,x^{0},u):1]$.

The proof of the following proposition is straightforward and we omit it.

\begin{proposition}
\label{Proposition_general}Consider in $\mathbb{R}^{n}$ the affine control
system (\ref{affine}), its homogeneous part (\ref{hom_n}), and in
$\mathbb{R}^{n+1}$ the homogeneous bilinear control system (\ref{hom_n+1}) as
well as the system in $\mathbb{P}^{n-1}$ induced by (\ref{hom_n}) and the
system (\ref{Pn}) in $\mathbb{P}^{n}$ induced by (\ref{hom_n+1}).

(i) Every control set $D\subset\mathbb{R}^{n}$ of the affine system
(\ref{affine}) yields a control set $\pi_{\mathbb{P}}D^{1}$ of the system
(\ref{Pn}) in $\mathbb{P}^{n}$ via the map%
\[
(x_{1},\ldots,x_{n})\longmapsto\lbrack x_{1}:\cdots:x_{n}:1]:\mathbb{R}%
^{n}\rightarrow\mathbb{P}^{n,1}\subset\mathbb{P}^{n}.
\]
Furthermore, $D$ is an invariant control set if and only if $\pi_{\mathbb{P}%
}D^{1}$ is an invariant control set. The control set $D$ is unbounded if and
only $\partial\left(  \pi_{\mathbb{P}}D^{1}\right)  \cap\mathbb{P}^{n,0}%
\not =\varnothing$. More precisely, if $x^{k}\in D$ with $\left\Vert
x^{k}\right\Vert \rightarrow\infty$, then every cluster point $y$ of
$\frac{x^{k}}{\left\Vert x^{k}\right\Vert }$ satisfies, in homogeneous
coordinates,
\[
\lbrack x_{1}^{k_{i}}:\cdots:x_{n}^{k_{i}}:1]\rightarrow\lbrack y_{1}%
:\cdots:y_{n}:0]\in\mathbb{P}^{n,0}\text{ for a subsequence }k_{i}%
\rightarrow\infty\text{.}%
\]

(ii) Every control set $_{\mathbb{P}}D^{\hom}$ and every chain control set
$_{\mathbb{P}}E^{\hom}$ of the system in $\mathbb{P}^{n-1}$ induced by
(\ref{hom_n}) corresponds to a unique control set $e(_{\mathbb{P}}D^{\hom})$
and chain control set $e(_{\mathbb{P}}E^{\hom})$, resp., of the system
(\ref{Pn}) restricted to $\mathbb{P}^{n,0}$ and conversely via the map $e$.
The invariant control sets in $\mathbb{P}^{n-1}$ correspond to the invariant
control sets in $\mathbb{P}^{n,0}$.
\end{proposition}

We remark that the assertions in Proposition \ref{Proposition_general}(ii)
also hold, if the accessibility rank condition in $\mathbb{P}^{n-1}$ is
violated (this is the case in Example \ref{Example_counter1}).

The intersection $\partial\left(  \pi_{\mathbb{P}}D^{1}\right)  \cap
\mathbb{P}^{n,0}$ will be of relevance below. Hence we give it a suggestive name.

\begin{definition}
\label{Definition_infinity}For a control set $D\subset\mathbb{R}^{n}$ with
associated control set $\pi_{\mathbb{P}}D^{1}$ in $\mathbb{P}^{n,1}$ the set
$\partial_{\infty}(D):=\partial\left(  \pi_{\mathbb{P}}D^{1}\right)
\cap\mathbb{P}^{n,0}$ is the boundary at infinity of $D$.
\end{definition}

Proposition \ref{Proposition_general}(i) shows, in particular, that the
boundary at infinity $\partial_{\infty}(D)$ is nonvoid if and only if $D$ is unbounded.

\begin{remark}
The construction of the boundary at infinity of a control set bears some
similarity to the ideal boundary used by Firer and do Rocio \cite{FidoR03} in
the analysis of invariant control sets for sub-semigroups of a semisimple Lie group.
\end{remark}

Next we clarify the relations between the accessibility rank conditions on the
relevant spaces.

\begin{theorem}
\label{Theorem_submanifolds}(i) If the accessibility rank condition holds for
affine system (\ref{affine}) on $\mathbb{R}^{n}$, then it also holds for the
system on the submanifold $\mathbb{P}^{n,1}\subset\mathbb{P}^{n}$ induced by
the bilinear system (\ref{hom_n+1}) on $\mathbb{R}^{n+1}$.

(ii) If the accessibility rank condition holds for the system on
$\mathbb{P}^{n-1}$ induced by the homogeneous part (\ref{hom_n}) of system
(\ref{affine}), then it holds for the system on the invariant submanifold
$\mathbb{P}^{n,0}\subset\mathbb{P}^{n}$ induced by the bilinear system
(\ref{hom_n+1}) on $\mathbb{R}^{n+1}$.
\end{theorem}

\begin{proof}
The proof is based on the local coordinate description of vector fields in
projective space obtained by projection of linear vector fields; cf. Bacciotti
and Vivalda \cite[Section 4]{BacV13}. We omit the details.
\end{proof}

\section{Control sets for nonhyperbolic systems\label{Section7}}

This section shows that all control sets with nonvoid interior are unbounded
if the hyperbolicity condition specified in Definition
\ref{Definition_hyperbolic} is violated. Using the compactification of the
state space constructed in the previous section, we show that there is a
single chain control set in $\mathbb{P}^{n}$ containing the images of all
control sets $D$ with nonvoid interior in $\mathbb{R}^{n}$ and the boundary at
infinity of this chain control set contains all chain control sets of the
homogeneous part having nonvoid intersection with the boundary at infinity of
one of the control sets $D$.

We begin with the following motivation. Consider a linear control system
\begin{equation}
\dot{x}(t)=Ax(t)+Bu(t),\quad u(t)\in\Omega, \label{linear}%
\end{equation}
where the control range $\Omega\subset\mathbb{R}^{m}$ is a compact convex
neighborhood of the origin. This is a special case of system (\ref{affine})
for $B_{1}=\cdots=B_{m}=0$ and $d=0$. We assume that the system without
control restriction is controllable. By Colonius and Kliemann \cite[Example
3.2.16]{ColK00} there is a unique control set $D$ with nonvoid interior, and
$0\in\mathrm{int}(D)$. Let $\mathbf{GE}(A;\mu)$ denote the real generalized
eigenspace for an eigenvalue $\mu$ of $A$. Then%
\begin{equation}
\mathbf{E}^{0}\subset D\subset\overline{K}+\mathbf{E}^{0}+F,
\label{linear_control__set}%
\end{equation}
where $\mathbf{E}^{0}:=\bigoplus_{\operatorname{Re}\mu=0}\mathbf{GE}(A;\mu)$
is the central spectral subspace and the sets $K\subset\mathbf{E}%
^{+}:=\bigoplus_{\operatorname{Re}\mu<0}\mathbf{GE}(A;\mu)$ and $F\subset
\mathbf{E}^{-}:=\bigoplus_{\operatorname{Re}\mu>0}\mathbf{GE}(A;\mu)$ are
bounded. This follows from Sontag \cite[Corollary 3.6.7]{Son98} showing that
$\mathcal{O}^{+}(0)=K+\mathbf{E}_{0}$. Then time reversal yields
$\mathcal{O}^{-}(0)=\mathbf{E}_{0}+F$ and hence, by Remark \ref{Remark2.2},%
\[
D=\overline{\mathcal{O}^{+}(0)}\cap\mathcal{O}^{-}(0)=\left(  \overline
{K}+\mathbf{E}^{0}\right)  \cap\left(  \mathbf{E}^{0}+F\right)  .
\]
Due to the decomposition $\mathbb{R}^{n}=\mathbf{E}^{+}\oplus\mathbf{E}%
^{0}\oplus\mathbf{E}^{-}$ this implies (\ref{linear_control__set}). In
particular, $D$ is bounded if and only if $\mathbf{E}^{0}=\{0\}$, i.e., if $A$
is a hyperbolic matrix. If $A$ is nonhyperbolic we embed system (\ref{linear})
into a homogeneous bilinear control system in $\mathbb{R}^{n+1}$ as explained
in Section \ref{Section6} and find that the boundary at infinity satisfies%
\begin{equation}
\partial_{\infty}(D)=\partial\left(  \pi_{\mathbb{P}}D^{1}\right)
\cap\mathbb{P}^{n,0}=\{[x_{1}:\cdots:x_{n}:0]\left\vert [x_{1}:\cdots
:x_{n}]\in\pi_{\mathbb{P}}\mathbf{E}^{0}\right.  \}. \label{infinity_linear}%
\end{equation}
This follows from (\ref{linear_control__set}) noting that for $0\not =%
x\in\mathbf{E}^{0}$ and every $j\in\mathbb{N}$ one obtains an element of $D$
given by $k_{j}+jx+f_{j}$ with $k_{j}\in\overline{K}$ and $f_{j}\in F$.
Considering the homogeneous coordinates and dividing by $j$ one finds for
$j\rightarrow\infty$ that (\ref{infinity_linear}) holds. The set
$\pi_{\mathbb{P}}\mathbf{E}^{0}$ is a maximal invariant chain transitive set
for the flow induced by the homogeneous part $\dot{x}=Ax$ on $\mathbb{P}%
^{n-1}$ (cf. Colonius and Kliemann \cite[Theorem 4.1.3]{ColK14}). Thus the
boundary at infinity $\partial_{\infty}(D)$ is a maximal invariant chain
transitive set for the induced flow on $\mathbb{P}^{n,0}$.

For general affine control systems of the form (\ref{affine}) it stands to
reason to replace the maximal chain transitive set $\pi_{\mathbb{P}}%
\mathbf{E}^{0}$ by maximal chain transitive sets of the control flow in
$\mathcal{U}\times\mathbb{P}^{n-1}$ associated with the homogeneous part or,
equivalently, by chain control sets in $\mathbb{P}^{n-1}$ (cf.
(\ref{chain_transitive2})) and to replace the spectral property of
$\mathbf{E}^{0}$ by appropriate generalized spectral properties. However, the
situation for general affine control systems will turn out to be more
intricate than for linear control systems.

Now we start our discussion of the nonhyperbolic case for (\ref{affine}). Here
several control sets with nonvoid interior may coexist as illustrated by
Example \ref{Example7.14} and \cite[Example 5.16 and Example 5.17]{ColRS}. The
following theorem shows that in the nonhyperbolic case all control sets with
nonvoid interior are unbounded.

\begin{theorem}
\label{Theorem_main2}Assume that the affine control system (\ref{affine}) on
$\mathbb{R}^{n}$ is nonhyperbolic.

(i) If there is a $\sigma$-periodic control $u\in\mathcal{U}_{pc}$ with
$g(u)\in\mathcal{S}_{\sigma}\cap\mathrm{int}(\mathcal{S})$ and $1\not \in
\mathrm{spec}(\Phi_{u}(\sigma,0))$, then there exists a control set $D$ with
nonvoid interior.

(ii) Every control set $D$ with nonvoid interior is unbounded. More precisely,
there are $x^{k}\in\mathrm{int}(D),k\in\mathbb{N}$ and $g(v)\in\mathcal{S}%
_{\tau}\cap\mathrm{int}(\mathcal{S}),\tau>0$, with $1\in\mathrm{spec}(\Phi
_{v}(\tau,0))$ and
\begin{equation}
\left\Vert x^{k}\right\Vert \rightarrow\infty\text{ and }d\left(  \frac{x^{k}%
}{\left\Vert x^{k}\right\Vert },\mathbf{E}(\Phi_{v}(\tau,0);1)\right)
\rightarrow0\text{ for }k\rightarrow\infty. \label{7.4}%
\end{equation}

\end{theorem}

\begin{proof}
(i) By Lemma \ref{Lemma_periodic} there is a unique $\sigma$-periodic
trajectory of (\ref{affine}) for $u$. Proposition \ref{Proposition2.5}(ii)
implies that it is contained in the interior of a control set $D$.

(ii) Let $x$ be in the interior of a control set $D$. By Proposition
\ref{Proposition2.5}(i) there are $\sigma>0$ and $g(u)\in\mathcal{S}_{\sigma
}\cap\mathrm{int}(\mathcal{S})$ such that $g(u)x=x$. Then the $\sigma
$-periodic control $u$ yields the $\sigma$-periodic trajectory $\varphi
(\cdot,x,u)\subset\mathrm{int}(D)$ and hence
\[
\int_{0}^{\sigma}\Phi_{u}(\sigma,s)\left(  Cu(s)+d\right)  ds=(I-\Phi
_{u}(\sigma,0))x.
\]

\textbf{Case 1}: If $1\in\mathrm{spec}(\Phi_{u}(\sigma,0))$ the affine
subspace $Y=x+\mathbf{E}(\Phi_{u}(\sigma,0);1)$ is contained in $\mathrm{int}%
(D)$. For the proof, an application of Lemma \ref{Lemma_image} shows that
there is a\ $\sigma$-periodic solution of (\ref{affine_0}) starting in $y$ if
and only if $y\in Y=x+\mathbf{E}(\Phi_{u}(\sigma,0);1)$. Thus $g(u)y=y\,$\ for
all $y\in Y$. Proposition \ref{Proposition2.5}(ii) implies that every $y$ is
in the interior of some control set, hence $Y\subset\mathrm{int}(D)$.
Furthermore, Lemma \ref{Lemma_image} also yields points $x^{k}\in Y$ such that
assertion (\ref{7.4}) holds with $v:=u$ and $\tau:=\sigma$.

\textbf{Case 2}: Suppose that $1\not \in \mathrm{spec}(\Phi_{u}(\sigma,0))$.
Since the system is nonhyperbolic there is a $\tau^{\ast}$-periodic control
$v^{\ast}$ with $1\in\mathrm{spec}(\Phi_{v^{\ast}}(\tau^{\ast},0))$ and
$g(v^{\ast})\in\mathcal{S}_{\tau^{\ast}}\cap\mathrm{int}(\mathcal{S})$.

Consider the continuous paths $p$ and $p^{\hom}$ from $g(u)$ to $g(v^{\ast})$
and $\Phi_{u}(\sigma,0)$ to $\Phi_{v^{\ast}}(\tau^{\ast},0)$, resp., given by
Lemma \ref{Lemma_path_a2}. Let%
\[
\alpha_{0}:=\sup\{\alpha\in\lbrack0,1]\left\vert \forall\alpha^{\prime}%
\in\lbrack0,\alpha):1\not \in \mathrm{spec}(\Phi_{u^{\alpha^{\prime}}}%
(\tau_{\alpha^{\prime}},0))\right.  \}.
\]
Hence, for $\alpha\in\lbrack0,\alpha_{0})$, Proposition
\ref{Proposition_periodicODE}(i) shows that there are unique $\tau_{\alpha}%
$-periodic trajectories for $u^{\alpha}$ which by Proposition
\ref{Proposition2.5}(ii) are in the interior of a control set. By Proposition
\ref{Proposition_analytic_new}(ii) $x^{\alpha}$ depends continuously on
$\alpha\in\lbrack0,\alpha_{0})$ and then the arguments in the proof of Theorem
\ref{Theorem_hyperbolic1} shows that the initial values satisfy $x^{\alpha}%
\in\mathrm{int}(D)$ for all $\alpha\in\lbrack0,\alpha_{0})$.

Now consider a sequence $\alpha_{k}\rightarrow\alpha_{0}$ with $\alpha
_{k}<\alpha_{0}$. Suppose first that%
\begin{equation}
\int_{0}^{\tau_{\alpha_{0}}}\Phi_{u^{\alpha_{0}}}(\tau_{\alpha_{0}},s)\left(
Cu^{\alpha_{0}}(s)+d\right)  ds\in\operatorname{Im}(I-\Phi_{u^{\alpha_{0}}%
}(\tau_{\alpha_{0}},0)).\label{7.4b}%
\end{equation}
Let for $k=0,1,2,\ldots$%
\[
b_{k}:=\int_{0}^{\tau_{\alpha_{k}}}\Phi_{u^{\alpha_{k}}}(\tau_{\alpha_{k}%
},s)\left(  Cu^{\alpha_{k}}(s)+d\right)  ds,\quad A_{k}:=I-\Phi_{u^{\alpha
_{k}}}(\tau_{\alpha_{k}},0).
\]
Then $A_{k}x^{\alpha_{k}}=b_{k}$ and $A_{k}\rightarrow A_{0},b_{k}\rightarrow
b_{0}$ for $k\rightarrow\infty$, and $\mathrm{\ker}A_{0}=\mathbf{E}%
(\Phi_{u^{\alpha_{0}}}(\tau_{\alpha_{0}},0);1)$. If $x^{\alpha_{k}}$ remains
bounded, we may assume that $x^{\alpha_{k}}\rightarrow x^{0}$ for some
$x^{0}\in\mathbb{R}^{n}$ and hence $A_{0}x^{0}=b_{0}$. Since $1\in
\mathrm{spec}(\Phi_{u^{\alpha_{0}}},0))$ Lemma \ref{Lemma_image} implies
assertion (\ref{7.4}) for $v:=u^{\alpha_{0}}$, as in Case 1. If $x^{\alpha
_{k}}$ becomes unbounded then%
\[
\left\Vert A_{0}\frac{x^{\alpha_{k}}}{\left\Vert x^{\alpha_{k}}\right\Vert
}\right\Vert \leq\left\Vert A_{0}-A_{k}\right\Vert +\left\Vert A_{k}%
\frac{x^{\alpha_{k}}}{\left\Vert x^{\alpha_{k}}\right\Vert }\right\Vert
\leq\left\Vert A_{0}-A_{k}\right\Vert +\frac{b_{k}}{\left\Vert x^{\alpha_{k}%
}\right\Vert }\rightarrow0\text{ for }k\rightarrow\infty,
\]
and again (\ref{7.4}) follows.

If (\ref{7.4b}) does not hold, Lemma \ref{Lemma_path3} implies that, for
$k=1,2,\ldots$, the initial values $x^{k}$ of the $\tau_{\alpha_{k}}$-periodic
solutions satisfy (\ref{7.4}) with $v:=u^{\alpha_{0}},\tau=\tau_{\alpha_{0}}$.
\end{proof}

Next we discuss the relation of the boundary at infinity to control sets of
the homogeneous part of the affine control system, motivated by the case of
linear control systems exposed in the beginning of this section. First we
obtain the following result for invariant control sets.

\begin{theorem}
\label{Theorem_inva}Assume that the affine system (\ref{affine}) is
nonhyperbolic and suppose that $D$ is an invariant control set.

(i) Then the interior of $D$ is nonvoid, the set $D$ is unbounded in
$\mathbb{R}^{n}$, and the boundary at infinity\ $\partial_{\infty}(D)$
contains an invariant control set $e(_{\mathbb{P}}D^{\hom})$ of the system
restricted to $\mathbb{P}^{n,0}$.

(ii) If the control range $\Omega$ is a compact convex neighborhood of the
origin and the system on $\mathbb{P}^{n-1}$ satisfies the accessibility rank
condition, then the boundary at infinity $\partial_{\infty}(D)$ contains the
unique invariant control set $e(_{\mathbb{P}}D^{\hom})$ where $_{\mathbb{P}%
}D^{\hom}$ the unique invariant control set on $\mathbb{P}^{n-1}$.
\end{theorem}

\begin{proof}
(i) By local accessibility, the interior of the invariant control set $D$ is
nonvoid and hence Theorem \ref{Theorem_main2} shows that $D$ is unbounded. It
follows that there is a point $\pi_{\mathbb{P}}(x,0)\in$ $\overline
{\pi_{\mathbb{P}}D^{1}}\cap\mathbb{P}^{n,0}$ and by Proposition
\ref{Proposition_general}(i) $\pi_{\mathbb{P}}D^{1}$ is an invariant control
set contained in $\mathbb{P}^{n,1}$. Since $\mathbb{P}^{n,0}$ is invariant and
closed in $\mathbb{P}^{n}$ it follows that $\overline{\mathcal{O}^{+}%
(\pi_{\mathbb{P}}(x,0))}\subset\mathbb{P}^{n,0}$, hence every point in this
set has the form $\pi_{\mathbb{P}}(y,0)$.\ For fixed $\varepsilon>0$ there are
$T>0$ and $u\in\mathcal{U}$ with $d(\pi_{\mathbb{P}}\psi(T,(x,0),u),\pi
_{\mathbb{P}}(y,0))<\varepsilon$ (recall (\ref{hom0}) and (\ref{hom1})).
Furthermore, there are $\pi_{\mathbb{P}}(x^{k},1)\in\,\pi_{\mathbb{P}}D^{1}$
with $\pi_{\mathbb{P}}(x^{k},1)\rightarrow\pi_{\mathbb{P}}(x,0)$. Since
$\pi_{\mathbb{P}}D^{1}$ is an invariant control set it follows that
$\pi_{\mathbb{P}}\psi(T,(x^{k},1),u)\in\overline{\pi_{\mathbb{P}}D^{1}}$, and
continuous dependence on the initial values implies that $d(\pi_{\mathbb{P}%
}\psi(T,(x^{k},1),u),\pi_{\mathbb{P}}(y,0))<\varepsilon$ for $k$ large enough.
Since $\varepsilon>0$ is arbitrary, we have shown that $\pi_{\mathbb{P}%
}(y,0)\in\overline{\pi_{\mathbb{P}}D^{1}}$ and hence $\overline{\mathcal{O}%
^{+}(\pi_{\mathbb{P}}(x,0))}\subset\partial\left(  \pi_{\mathbb{P}}%
D^{1}\right)  \cap\mathbb{P}^{n,0}=\partial_{\infty}(D)$. By Colonius and
Kliemann \cite[Theorem 3.2.8]{ColK00}, for every point $\pi_{\mathbb{P}}(x,0)$
in the compact space $\mathbb{P}^{n,0}$ there is an invariant control set
contained in the closure of the reachable set $\overline{\mathcal{O}^{+}%
(\pi_{\mathbb{P}}(x,0))}$. By Proposition \ref{Proposition_general}(ii) this
invariant control set has the form $e(_{\mathbb{P}}D^{\hom})$ implying (i).

(ii) This follows from Proposition \ref{Proposition_general}(ii) since the
accessibility rank condition implies by Theorem \ref{Theorem_95}(i) that the
invariant control set of the system on $\mathbb{P}^{n-1}$ is unique.
\end{proof}

We proceed to prove the following result on the relation between the boundary
at infinity of a control set in $\mathbb{R}^{n}$ and the chain control sets of
the homogeneous part in $\mathbb{P}^{n-1}$.

\begin{proposition}
\label{Theorem_main3}Assume that the affine control system (\ref{affine}) on
$\mathbb{R}^{n}$ is nonhyperbolic. Then for every control set $D\subset
\mathbb{R}^{n}$ with nonvoid interior of (\ref{affine}) there is a chain
control set $_{\mathbb{P}}E^{\hom}\subset\mathbb{P}^{n-1}$ such that
$\partial_{\infty}(D)\cap e(_{\mathbb{P}}E^{\hom})\not =\varnothing$.
\end{proposition}

\begin{proof}
Theorem \ref{Theorem_main2} shows that $D$ is unbounded and that there are
$g(v)\in\mathcal{S}_{\tau}\cap\mathrm{int}(\mathcal{S})$ with $1\in
\mathrm{spec}(\Phi_{v}(\tau,0))$ and $x^{k}\in\mathrm{int}(D)$ satisfying
$\left\Vert x^{k}\right\Vert \rightarrow\infty$ and $d(\frac{x^{k}}{\left\Vert
x^{k}\right\Vert },\mathbf{E}(\Phi_{v}(\tau,0);1))\rightarrow0$ for
$k\rightarrow\infty$. Since $x=\Phi_{v}(\tau,0)x$ for all $x\in\mathbf{E}%
(\Phi_{v}(\tau,0);1)$ the path connected set $\pi_{\mathbb{P}}\mathbf{E}%
(\Phi_{v}(\tau,0);1)$ consists of points on $\tau$-periodic solutions for the
$\tau$-periodic control $v$ and hence is contained in a chain control set
$_{\mathbb{P}}E^{\hom}$. Hence one obtains $\partial_{\infty}(D)\cap
e(_{\mathbb{P}}E^{\hom})\not =\varnothing$ in $\mathbb{P}^{n,0}$.
\end{proof}

The following theorem presents a partial converse of Theorem
\ref{Theorem_inva}(i).

\begin{theorem}
\label{Theorem_main4}Assume that the homogeneous part (\ref{hom_n}) of the
affine system (\ref{affine}) satisfies the accessibility rank condition on
$\mathbb{R}^{n}\setminus\{0\}$, that there is $g(u)\in S_{\sigma}%
\cap\mathrm{int}(\mathcal{S})$ with $1\not \in \mathrm{spec}(\Phi_{u}%
(\sigma,0))$ for some $\sigma>0$, and that there are at most finitely many
control sets with nonvoid interior of system (\ref{affine}). Then for every
control set $_{\mathbb{R}}D_{i}^{\hom}$ with nonvoid interior there exists
such a control set $D$ of (\ref{affine}) with boundary at infinity satisfying%
\begin{equation}
\partial_{\infty}\left(  D\right)  \cap e(_{\mathbb{P}}D_{i}^{\hom}%
)\not =\varnothing\text{ for }_{\mathbb{P}}D_{i}^{\hom}\supset\pi_{\mathbb{P}%
}(_{\mathbb{R}}D_{i}^{\hom}).\label{7.5d}%
\end{equation}

\end{theorem}

\begin{proof}
Fix a point $x\in\mathrm{int}(_{\mathbb{R}}D_{i}^{\hom})$. Since by Theorem
\ref{Theorem_generalLie} $\mathrm{int}(\mathcal{S}_{\leq\tau})\not =%
\varnothing$ for all $\tau>0$ there are $\tau_{0}>0$ small enough and
$u^{0}\in\mathcal{U}_{pc}$ with $g(u^{0})\in\mathcal{S}_{\tau_{0}}%
\cap\mathrm{int}(\mathcal{S})$ and $x^{0}:=\Phi_{u^{0}}(\tau_{0}%
,0)x\in\mathrm{int}(_{\mathbb{R}}D_{i}^{\hom})$. Since also $\mathrm{int}%
(_{\mathbb{R}}\mathcal{S}_{\leq\tau}^{\hom})\not =\varnothing$ for all
$\tau>0$ there are $\tau_{1}>0$ small enough and $u^{1}\in\mathcal{U}_{pc}$
such that the corresponding element $\Phi_{u^{1}}(\tau_{1},0)\in
\,_{\mathbb{R}}\mathcal{S}_{\tau_{1}}^{\hom}\cap\mathrm{int}(_{\mathbb{R}%
}\mathcal{S}^{\hom})$ satisfies%
\[
x^{1}:=\Phi_{u^{1}}(\tau_{1},0)x^{0}=\Phi_{u^{1}}(\tau_{1},0)\Phi_{u^{0}}%
(\tau_{0},0)x\in\mathrm{int}(_{\mathbb{R}}D_{i}^{\hom}).
\]
By Remark \ref{Remark2.2} controllability in the interior of $_{\mathbb{R}%
}D_{i}^{\hom}$ holds, hence there are $\tau_{2}>0$ and $u^{2}\in
\mathcal{U}_{pc}$ satisfying $\Phi_{u^{2}}(\tau_{2},0)x^{1}=x$. Define
$\tau:=\tau_{0}+\tau_{1}+\tau_{2}$ and a control $u\in\mathcal{U}_{pc}$ by
$\tau$-periodic extension of%
\[
u(t):=\left\{
\begin{array}
[c]{lll}%
u^{0}(t) & \text{for} & t\in\lbrack0,\tau_{0})\\
u^{1}(t-\tau_{0}) & \text{for} & t\in\lbrack\tau_{0},\tau_{0}+\tau_{1})\\
u^{2}(t-\tau_{0}-\tau_{1}) & \text{for} & t\in\lbrack\tau_{0}+\tau_{1}%
,\tau_{0}+\tau_{1}+\tau_{2})
\end{array}
\right.  .
\]
Then $\Phi_{u}(\tau,0)x=x$, hence $1\in\mathrm{spec}(\Phi_{u}(\tau,0))$, and
\begin{align*}
g(u) &  =g(u^{2})g(u^{1})g(u^{0})\in\mathcal{S}_{\tau}\cap\mathrm{int}%
(\mathcal{S}),\\
\Phi_{u}(\tau,0) &  =\Phi_{u^{2}}(\tau_{2},0)\Phi_{u^{1}}(\tau_{1}%
,0)\Phi_{u^{0}}(\tau_{0},0)\in\mathrm{int}(_{\mathbb{R}}\mathcal{S}^{\hom}).
\end{align*}
By Proposition \ref{Proposition2.5}(ii) (for $_{\mathbb{R}}\mathcal{S}%
^{\mathrm{hom}}$) this implies that the eigenspace $\mathbf{E}(\Phi_{u}%
(\tau,0);1)$ of $\Phi_{u}(\tau,0)$ for the eigenvalue $1$ is contained in the
interior of some control set in $\mathbb{R}^{n}\setminus\{0\}$. Since
$x\in\mathbf{E}(\Phi_{u}(\tau,0);1)\cap\,_{\mathbb{R}}D_{i}^{\hom}$ it follows
that
\begin{equation}
\mathbf{E}(\Phi_{u}(\tau,0);1)\subset\mathrm{int}(_{\mathbb{R}}D_{i}^{\hom
})\text{ and hence }\pi_{\mathbb{P}}\mathbf{E}(\Phi_{u}(\tau,0);1)\subset
\mathrm{int}(_{\mathbb{P}}D_{i}^{\hom}).\label{7.5a}%
\end{equation}
By Proposition \ref{Proposition_analytic_new}(i) there are $g(u^{k}%
)\in\mathcal{S}_{\tau_{k}}\cap\mathrm{int}(\mathcal{S})$ with $1\not \in
\mathrm{spec}(\Phi_{u^{k}}(\tau_{k},0))$ and $\Phi_{u^{k}}(\tau_{k}%
,0)\rightarrow\Phi_{u}(\tau,0)$ for $k\rightarrow\infty$. Proposition
\ref{Proposition_periodicODE}(i) implies that there are unique $\tau_{k}%
$-periodic solutions denoted by $\varphi(\cdot,x^{k},u^{k})$ of the affine
equation for the $\tau_{k}$-periodic extension of $u^{k}$. By Proposition
\ref{Proposition2.5}(ii) each of them is in the interior of a control set for
the affine system (\ref{affine}). Since, by assumption, there are only
finitely many of them, infinitely many $x^{k}$ are contained in the interior
of a single control set $D$. We may assume that all $x^{k}$ are in
$\mathrm{int}(D)$.

Suppose that assumption (iii) in Lemma \ref{Lemma_path3} is satisfied. Then it
follows that%
\begin{equation}
\left\Vert x^{k}\right\Vert \rightarrow\infty\text{ and }\frac{x^{k}%
}{\left\Vert x^{k}\right\Vert }\rightarrow\mathbf{E}(\Phi_{u}(\tau,0);1)\text{
for }k\rightarrow\infty. \label{7.5b}%
\end{equation}
Hence, for $k\rightarrow\infty$, the points $\pi_{\mathbb{P}}(x^{k},1)\in
\,\pi_{\mathbb{P}}D^{1}$ converge to $e\left(  \pi_{\mathbb{P}}\mathbf{E}%
(\Phi_{u}(\tau,0);1)\right)  $ showing that%
\[
\overline{\pi_{\mathbb{P}}D^{1}}\cap e\left(  \pi_{\mathbb{P}}\mathbf{E}%
(\Phi_{u}(\tau,0);1)\right)  \not =\varnothing.
\]
Together with (\ref{7.5a}) this implies that the boundary at infinity of $D$
satisfies (\ref{7.5d}).

If assumption (iii) in Lemma \ref{Lemma_path3} does not hold Lemma
\ref{Lemma_image} shows that there are $x^{k}\in\mathbf{E}(\Phi_{u}%
(\tau,0);1)$ with (\ref{7.5b}) for $k\rightarrow\infty$. Then the assertion
follows also in this case.
\end{proof}

The following examples (cf. Mohler \cite[Example 2 on page 32]{Mohler} and
Colonius, Santana, Setti \cite[Example 5.16]{ColRS}) show that, in general,
the boundary at infinity of a control set $D$ may intersect more than one
control set of the projectivized homogeneous part.

\begin{example}
\label{Example_counter1}Consider the affine control system%
\[
\left(
\begin{array}
[c]{c}%
\dot{x}\\
\dot{y}%
\end{array}
\right)  =\left(
\begin{array}
[c]{cc}%
2u & 1\\
1 & 2u
\end{array}
\right)  \left(
\begin{array}
[c]{c}%
x\\
y
\end{array}
\right)  +\left(
\begin{array}
[c]{c}%
0\\
1
\end{array}
\right)  u,\quad u(t)\in\Omega=[-1,1].
\]
The eigenvalues of $A(u)=A+uB$ are given by $\lambda_{1}(u)=2u+1>\lambda
_{2}(u)=2u-1$ and $\lambda_{1}(-1/2)=\lambda_{2}(1/2)=0$. For every
$u\in\mathbb{R}$, the eigenspaces for $\lambda_{1}(u)$ and $\lambda_{2}(u)$
are $\mathbf{E}(A+uB;\lambda_{1}(u))=\{(z,z)^{\top}\left\vert z\in
\mathbb{R}\right.  \}$ and $\mathbf{E}(A+uB;\lambda_{2}(u))=\{(z,-z)^{\top
}\left\vert z\in\mathbb{R}\right.  \}$, resp. In the northern part of the unit
circle (hence in $\mathbb{P}^{1}$) this yields the two one-point control sets
given by the equilibria for any $u\in\lbrack-1,1]$,%
\[
_{\mathbb{P}}D_{1}^{\hom}=\left\{  \left(  1/\sqrt{2},1/\sqrt{2}\right)
^{\top}\right\}  \text{ and }_{\mathbb{P}}D_{2}^{\hom}=\left\{  \left(
-1/\sqrt{2},1/\sqrt{2}\right)  ^{\top}\right\}  .
\]
Here the trajectories not starting in one of these equilibria approach
$_{\mathbb{P}}D_{1}^{\hom}$ and $_{\mathbb{P}}D_{2}^{\hom}$ for $t\rightarrow
+\infty$ and $t\rightarrow-\infty$, resp. As shown in \cite[Example
4.4/5.16]{ColRS} there is a connected branch of equilibria of the affine
system
\[
\mathcal{B}_{1}=\left\{  \left(  x_{u},y_{u})^{\top}\right)  \left\vert
u\in\left(  -1/2,1/2\right)  \right.  \right\}  \text{ with }\left(
x_{0},y_{0}\right)  ^{\top}=\left(  0,0\right)  ^{\top}\in\mathcal{B}_{1}.
\]
They become unbounded for $\left\vert u\right\vert \rightarrow\frac{1}{2}$,
and there is a single control set $D$ containing the equilibria in
$\mathcal{B}_{1}$ in the interior. The equilibria in $\mathcal{B}_{1}$ satisfy
$\frac{(x_{u},y_{u})}{\left\Vert (x_{u},y_{u})\right\Vert }\rightarrow\left(
\mp\frac{1}{\sqrt{2}},\frac{1}{\sqrt{2}}\right)  $ for $u\rightarrow\pm
\frac{1}{2}$. Consequently, one obtains for the control sets of the
homogeneous part%
\[
e(_{\mathbb{P}}D_{1}^{\hom})\cup\,e(_{\mathbb{P}}D_{2}^{\hom})\subset
\partial(\pi_{\mathbb{P}}D^{1})\cap\mathbb{P}^{2,0}=\partial_{\infty}(D).
\]

\end{example}

The homogeneous part of Example \ref{Example_counter1} violates the
accessibility rank condition in $\mathbb{P}^{1}$ and the control sets
$_{\mathbb{P}}D_{1}^{\hom}$ and $_{\mathbb{P}}D_{2}^{\hom}$ in $\mathbb{P}%
^{1}$ have void interiors. We modify this example in order to get control sets
in $\mathbb{P}^{1}$ with nonvoid interior. Note that here an arbitrarily small
perturbation suffices to change the system behavior drastically.

\begin{example}
\label{Example_counter2}Consider for small $\varepsilon>0$ and $\Omega=[-1,1]$%
\[
\left(
\begin{array}
[c]{c}%
\dot{x}\\
\dot{y}%
\end{array}
\right)  =\left(
\begin{array}
[c]{cc}%
2u & 1\\
1 & (2+\varepsilon)u
\end{array}
\right)  \left(
\begin{array}
[c]{c}%
x\\
y
\end{array}
\right)  +\left(
\begin{array}
[c]{c}%
0\\
1
\end{array}
\right)  u=\left(  A+uB(\varepsilon)\right)  \left(
\begin{array}
[c]{c}%
x\\
y
\end{array}
\right)  +Cu,
\]
We will show that there is a control set $D$ in $\mathbb{R}^{2}$ such that the
boundary at infinity $\partial_{\infty}(D)$ intersects two control sets with
nonvoid interior for the homogeneous part.

\textbf{Step 1:} The eigenvalues of $A+uB(\varepsilon)$ are given by%
\[
\lambda_{1,2}(u,\varepsilon)=\frac{u}{2}(4+\varepsilon)\pm\frac{1}{2}%
\sqrt{4+u^{2}\left[  (4+\varepsilon)^{2}-4(4+2\varepsilon)\right]  }.
\]
Note that $\lambda_{1}(u,\varepsilon)>\lambda_{2}(u,\varepsilon)$ for all
$u\in\lbrack-1,1]$. For $\varepsilon=0$, it is clear that the functions
$u\mapsto\lambda_{1,2}(u,0)=2u\pm1$ are strictly increasing, hence this also
holds for small $\varepsilon>0$. Thus there are unique values $u^{1}%
(\varepsilon),u^{2}(\varepsilon)\in(-1,1)$ with $\lambda_{1}(u^{1}%
(\varepsilon),\varepsilon)=0$ and $\lambda_{2}(u^{2}(\varepsilon
),\varepsilon)=0$, and $u^{1}(\varepsilon)\rightarrow-\frac{1}{2}$ and
$u^{2}(\varepsilon)\rightarrow\frac{1}{2}$ for $\varepsilon\rightarrow0$. The
eigenvectors $(x,y)^{\top}$ satisfy $y=\left(  \lambda_{1,2}(u,\varepsilon
)-2u\right)  x$. For $\varepsilon\rightarrow0$ and all $u\in\lbrack-1,1]$ the
eigenspace $\mathbf{E}(A+uB(\varepsilon);\lambda_{i}(u,\varepsilon))$
converges to the eigenspace $\mathbf{E}(A+uB(0);\lambda_{i}(u,0))$. In the
northern part of the unit circle (hence in $\mathbb{P}^{1}$) this yields two
equilibria $e_{1}(u,\varepsilon)$ and $e_{2}(u,\varepsilon)$, and the other
trajectories in $\mathbb{P}^{1}$ converge for $t\rightarrow\infty$ to
$e_{1}(u,\varepsilon)$ and for $t\rightarrow-\infty$ to $e_{2}(u,\varepsilon
)$. Hence there are control sets $_{\mathbb{P}}D_{1}^{\hom}$ and
$_{\mathbb{P}}D_{2}^{\hom}$ (depending on $\varepsilon$) with nonvoid interior
consisting of the equilibria $e_{1}(u,\varepsilon)$ and $e_{2}(u,\varepsilon
),u\in\lbrack-1,1]$, resp. The control set $_{\mathbb{P}}D_{1}^{\hom}$ is
invariant. One easily verifies the accessibility rank condition in
$\mathbb{R}^{2}\setminus0\}$. Since $0\in\mathrm{int}(\Sigma_{Fl}%
(_{\mathbb{P}}D_{i}^{\hom}))$ it follows that $_{\mathbb{P}}D_{i}^{\hom}$ is
the projection to $\mathbb{P}^{1}$ of a control set $_{\mathbb{R}}D_{i}^{\hom
}$ in $\mathbb{R}^{2}\setminus\{0\}$, $i=1,2$.

\textbf{Step 2:} The equilibria $(x_{u}(\varepsilon),y_{u}(\varepsilon
))^{\top}$ approach $\mathbf{E}(A+u^{i}(\varepsilon)B(\varepsilon);0)$ for
$u\rightarrow u^{i}(\varepsilon),i=1,2$. In both cases, the equilibria become
unbounded. In particular, there is a connected unbounded branch of equilibria
\[
\mathcal{B}_{1}(\varepsilon)=\left\{  (x_{u}(\varepsilon),y_{u}(\varepsilon
))^{\top}\left\vert u\in\left(  u^{1}(\varepsilon),u^{2}(\varepsilon)\right)
\right.  \right\}
\]
and a single control set $D$ (again depending on $\varepsilon$) containing the
equilibria in $\mathcal{B}_{1}(\varepsilon)$.

\textbf{Step 3:} Embedding the control system into a homogeneous bilinear
system in $\mathbb{R}^{3}$ and projecting it to $\mathbb{P}^{2}$ one obtains
from the control set $D$ a control set in $\mathbb{P}^{2,1}$ given by
$\pi_{\mathbb{P}}D^{1}=\left\{  [x:y:1]\left\vert (x,y)^{\top}\in D\right.
\right\}  $. As the equilibria $(x_{u}(\varepsilon),y_{u}(\varepsilon))^{\top
}\in\mathcal{B}_{1}$ become unbounded for $u\rightarrow u^{i}(\varepsilon)$
they approach the eigenspace $\mathbf{E}(A+u^{i}(\varepsilon)B(\varepsilon
);0)$, hence%
\[
e(_{\mathbb{P}}D_{1}^{\hom})\cap\partial_{\infty}(D)\not =\varnothing\text{
and }e(_{\mathbb{P}}D_{2}^{\hom})\cap\partial_{\infty}(D)\not =\varnothing.
\]

\end{example}

In the following we require that $\Omega$ is a convex and compact neighborhood
of $0\in\mathbb{R}^{m}$ and consider chain control sets of the affine system
in $\mathbb{P}^{n}$.

\begin{definition}
The boundary at infinity of a chain control set $_{\mathbb{P}}E$ for the
affine system (\ref{Pn}) in $\mathbb{P}^{n}$ is $\partial_{\infty
}(_{\mathbb{P}}E):=\partial(_{\mathbb{P}}E)\cap\mathbb{P}^{n,0}$.
\end{definition}

This definition is similar to the boundary at infinity for control sets but it
refers to chain control sets in $\mathbb{P}^{n}$ not requiring that they are
obtained from chain control sets in $\mathbb{R}^{n}$.

\begin{lemma}
\label{Lemma_chains}Let $_{\mathbb{P}}E$ be a chain control set in
$\mathbb{P}^{n}$.

(i) If $\partial_{\infty}(_{\mathbb{P}}E)\cap e(_{\mathbb{P}}E_{j}^{\hom
})\not =\varnothing$ for a chain control set $_{\mathbb{P}}E_{j}^{\hom}$ in
$\mathbb{P}^{n-1}$ of the homogeneous part, then $e(_{\mathbb{P}}E_{j}^{\hom
})\subset\partial_{\infty}(_{\mathbb{P}}E)$.

(ii) If $\partial_{\infty}(_{\mathbb{P}}E)$ is nonvoid, it contains a chain
control set $e(_{\mathbb{P}}E_{j}^{\hom})$ for a chain control set
$_{\mathbb{P}}E_{j}^{\hom}$ of the homogeneous part.
\end{lemma}

\begin{proof}
(i) Recall from Proposition \ref{Proposition_general}(ii) that $e(_{\mathbb{P}%
}E_{j}^{\hom})$ is a chain control set of the system restricted to
$\mathbb{P}^{n,0}$. We will show that the set $_{\mathbb{P}}E^{\prime
}:=\,_{\mathbb{P}}E\cup e(_{\mathbb{P}}E_{j}^{\hom})$ satisfies the properties
(i) and (ii) of a chain control set in $\mathbb{P}^{n}$. Then the maximality
property (iii) of the chain control set $_{\mathbb{P}}E$ implies that
$_{\mathbb{P}}E^{\prime}=\,_{\mathbb{P}}E$ showing that $e(_{\mathbb{P}}%
E_{j}^{\hom})\subset\partial_{\infty}(_{\mathbb{P}}E)$.

It is clear that $_{\mathbb{P}}E^{\prime}$ satisfies (i), since this holds for
$_{\mathbb{P}}E$ and $e(_{\mathbb{P}}E_{j}^{\hom})$. For property (ii),
consider $x\in\,_{\mathbb{P}}E$ and $y\in e(_{\mathbb{P}}E_{j}^{\hom})$ and
$\varepsilon,T>0$. Fix $z\in\partial_{\infty}(_{\mathbb{P}}E)\cap
e(_{\mathbb{P}}E_{j}^{\hom})=\,_{\mathbb{P}}E\cap e(_{\mathbb{P}}E_{j}^{\hom
})$. There are controlled $(\varepsilon,T)$-chains $\zeta_{1}$ and $\zeta_{2}$
from $x$ to $z$ and from $z$ to $x$, resp. For the system restricted to
$\mathbb{P}^{n,0}$, there exist controlled $(\varepsilon,T)$-chains $\zeta
_{3}$ and $\zeta_{4}$ from $z$ to $y$ and from $y$ to $z$, resp. Then the
concatenations $\zeta_{3}\circ\zeta_{1}$ and $\zeta_{3}\circ\zeta_{4}$ are
controlled $(\varepsilon,T)$-chains from $x$ to $y$ and from $y$ to $x$, resp.
This concludes the proof of assertion (i) since $\varepsilon,T>0$ are arbitrary.

(ii) Let $x\in\partial_{\infty}(_{\mathbb{P}}E)$. Then there exists a control
$u\in\mathcal{U}$ with $\pi_{\mathbb{P}}\varphi(t,x,u)\in\partial_{\infty
}(_{\mathbb{P}}E)=\,_{\mathbb{P}}E\cap\mathbb{P}^{n,0}$ for all $t\geq0$ by
property (i) of chain control sets and invariance of $\mathbb{P}^{n,0}$. Since
$_{\mathbb{P}}E\cap\mathbb{P}^{n,0}=\partial(_{\mathbb{P}}E)\cap
\mathbb{P}^{n,0}$ is compact, it follows that the $\omega$-limit set%
\[
\omega_{\mathbb{P}}(u,x):=\left\{  \left.  y=\lim\nolimits_{k\rightarrow
\infty}\pi_{\mathbb{P}}\varphi(t_{k},x,u)\right\vert t_{k}\rightarrow
\infty\right\}  \subset\,_{\mathbb{P}}E\cap\mathbb{P}^{n,0}.
\]
is nonvoid. Hence Colonius and Kliemann \cite[Corollary 4.3.12]{ColK00}
implies that there exists a chain control set of the system restricted to
$\mathbb{P}^{n,0}$ containing $\omega_{\mathbb{P}}(u,x)$. Thus there is a
chain control set $_{\mathbb{P}}E_{j}^{\hom}$ in $\mathbb{P}^{n-1}$ of the
homogeneous part with $\partial_{\infty}(_{\mathbb{P}}E)\cap e(_{\mathbb{P}%
}E_{j}^{\hom})\not =\varnothing$. Now the assertion follows from (i).
\end{proof}

The next theorem is the main result on the control sets $D$ with nonvoid
interior in $\mathbb{R}^{n}$ in the nonhyperbolic case.

\begin{theorem}
\label{Theorem_main5}Assume that the affine control system (\ref{affine}) is
nonhyperbolic. Furthermore, let the control range $\Omega$ be a compact convex
neighborhood of the origin and assume that there is $g(u)\in S_{\sigma}%
\cap\mathrm{int}(\mathcal{S})$ with $1\not \in \mathrm{spec}(\Phi_{u}%
(\sigma,0))$ for some $\sigma>0$.

Then there exists a single chain control set $_{\mathbb{P}}E$ in
$\mathbb{P}^{n}$ containing the control sets $\pi_{\mathbb{P}}D^{1}$ for all
control sets $D$ with nonvoid interior in $\mathbb{R}^{n}$. Furthermore, the
boundary at infinity $\partial_{\infty}(_{\mathbb{P}}E)$ contains all
$\partial_{\infty}(D)$ and the chain control sets $e(_{\mathbb{P}}E_{j}^{\hom
})$ where $_{\mathbb{P}}E_{j}^{\hom}$ are the chain control sets in
$\mathbb{P}^{n-1}$ for the homogeneous part (\ref{hom_n}) with $\partial
_{\infty}(D)\cap e(_{\mathbb{P}}E_{j}^{\hom})\not =\varnothing$ for some $D$.
\end{theorem}

\begin{proof}
By Proposition \ref{Proposition_periodicODE}(i) there is a unique $\sigma
$-periodic solution of the affine system with $g(u)x^{0}=x^{0}$. Proposition
\ref{Proposition2.5} implies that $x^{0}$ is in the interior of a control set
$D_{0}$. Now let $D_{1}$ be any control set in $\mathbb{R}^{n}$ with nonvoid
interior. It suffices to show that there is a chain control set $_{\mathbb{P}%
}E$ in $\mathbb{P}^{n}$ containing $\pi_{\mathbb{P}}D_{0}^{1}$ and
$\pi_{\mathbb{P}}D_{1}^{1}$ and that its boundary at infinity $\partial
_{\infty}(_{\mathbb{P}}E)$ contains all chain control sets $e(_{\mathbb{P}%
}E_{j}^{\hom})$ with $\partial_{\infty}(D_{1})\cap e(_{\mathbb{P}}E_{j}^{\hom
})\not =\varnothing$.

Pick $x^{1}\in\mathrm{int}(D_{1})$. Then Proposition \ref{Proposition2.5}(i)
implies that there are $\tau_{1}>0$ and $g_{1}=g(u^{1})\in\mathcal{S}%
_{\tau_{1}}\cap\mathrm{int}(\mathcal{S})$ with $x^{1}=g(u^{1})x^{1}$.
Proposition \ref{Proposition_analytic_new}(i) yields $\tau_{\alpha}$-periodic
controls $w^{\alpha}\in\mathcal{U}_{pc}$ and continuous paths $p_{1}%
:[0,1]\rightarrow\mathrm{int}(\mathcal{S})$ with $p_{1}(0)=g(u),\,p_{1}%
(\alpha)=g(w^{\alpha})$ for $\alpha\in\lbrack0,1]$ and $p_{1}^{\hom
}:[0,1]\rightarrow\,_{\mathbb{R}}\mathcal{S}^{\hom}$ with $p_{1}^{\hom
}(0)=\Phi_{u}(\sigma,0),\,p_{1}^{\hom}(\alpha)=\Phi_{w^{\alpha}}(\tau_{\alpha
},0)$ such that%
\[
\left\Vert p_{1}^{\hom}(1)-\Phi_{u^{1}}(\tau_{1},0)\right\Vert <\varepsilon
,\,\left\Vert p_{1}(1)-g(u^{1})\right\Vert <\varepsilon
\]
and $1\not \in \mathrm{spec}(\Phi_{w^{\alpha}}(\tau_{\alpha},0)$ for all but
at most finitely many $\alpha\in\lbrack0,1]$. Denote the $\alpha$-values with
$1\in\mathrm{spec}(\Phi_{w^{\alpha}}(\tau_{\alpha},0))$ by $0<\alpha
_{1}<\cdots<\alpha_{r}<1$, where the last inequality holds without loss of
generality. The continuity and smoothness properties from Proposition
\ref{Proposition_analytic_new}(i) hold for $w^{\alpha},\tau_{\alpha}%
,\Phi_{w^{\alpha}}(\tau_{\alpha},0)\in\,_{\mathbb{R}}\mathcal{S}^{\hom}$, and
$g(w^{\alpha})\in\mathcal{S}$. For $\alpha$ with $1\not \in \mathrm{spec}%
(\Phi_{w^{\alpha}}(\tau_{\alpha},0)$ Proposition
\ref{Proposition_analytic_new}(ii) shows that there are unique $\tau_{\alpha}%
$-periodic solutions with initial values $x^{\alpha}$ depending continuously
on $\alpha$. Define%
\[
A:=\{\alpha\in\lbrack0,1]\left\vert \text{ }1\not \in \mathrm{spec}%
(\Phi_{w^{\alpha}}(\tau_{\alpha},0))\right.  \}.
\]
The set $A$ consists of $r+1$ intervals. For $\alpha\in A$, Proposition
\ref{Proposition2.5}(ii) implies that $x^{\alpha}\in\mathrm{int}(D^{\alpha})$
for some control set $D^{\alpha}$ since $g(w^{\alpha})\in\mathrm{int}%
(\mathcal{S})$. For each $\alpha$ in an interval contained in $A$, the
$x^{\alpha}$ depend continuously on $\alpha$, hence they are contained in the
interior of a single control set. By construction, $x^{0}\in\mathrm{int}%
(D_{0})$ and, for $\varepsilon>0$ small enough, $x^{1}\in\mathrm{int}(D_{1})$.
Denote the other control sets containing the $x^{\alpha}$ by $D_{i},i\geq2$.
By Theorem \ref{Theorem_main2} the control sets $D_{i}$ are unbounded, hence
their boundary at infinity $\partial_{\infty}(D_{i})$ is nonvoid.

The control sets $\pi_{\mathbb{P}}D_{i}^{1}$ in $\mathbb{P}^{n,1}$ are
contained in chain control sets $_{\mathbb{P}}E_{i}\subset\mathbb{P}^{n}$ and
it follows that $\partial_{\infty}(_{\mathbb{P}}E_{i})\supset\partial_{\infty
}(D_{i})$ is nonvoid. Thus Lemma \ref{Lemma_chains}(ii) implies that
$\partial_{\infty}(_{\mathbb{P}}E_{i})$ contains a chain control set
$e(_{\mathbb{P}}E_{j}^{\hom})$ where $_{\mathbb{P}}E_{j}^{\hom}$ is a chain
control set of the homogeneous part in $\mathbb{P}^{n-1}$. By Lemma
\ref{Lemma_chains}(i), $\partial_{\infty}(_{\mathbb{P}}E_{i})$ contains every
chain control set of the homogeneous part that it intersects. The theorem
follows from the next claim.

\textbf{Claim. }All chain control sets $_{\mathbb{P}}E_{i},i\geq2$, in
$\mathbb{P}^{n}$ coincide.

For every point $\alpha_{i}\in(0,1)\setminus A$ there are control sets which
we denote by $D_{i}$ and $D_{i+1}$ such that all $\alpha$ in a neighborhood of
$\alpha_{i}$ satisfy $x^{\alpha}\in\mathrm{int}(D_{i})$ for $\alpha<\alpha
_{i}$ and $x^{\alpha}\in\mathrm{int}(D_{i+1})$ for $\alpha>\alpha_{i}$.

We have to show that the chain control sets $_{\mathbb{P}}E_{i}$ for
$\alpha<\alpha_{i}$ and $_{\mathbb{P}}E_{i+1}$ for $\alpha>\alpha_{i}$
coincide. The projected eigenspace $\pi_{\mathbb{P}}\mathbf{E}(\Phi
_{u^{\alpha_{i}}}(\tau_{\alpha_{i}},0);1))$ consists of $\tau_{\alpha_{i}}%
$-periodic solutions for the $\tau_{\alpha_{i}}$-periodic control
$u^{\alpha_{i}}$ and hence is contained in a chain control set $_{\mathbb{P}%
}E_{j}^{\hom}$ in $\mathbb{P}^{n-1}$. We will show that
\[
e(_{\mathbb{P}}E_{j}^{\hom})\subset\partial_{\infty}(_{\mathbb{P}}E_{i}%
)\cap\partial_{\infty}(_{\mathbb{P}}E_{i+1}),
\]
which implies that $_{\mathbb{P}}E_{i}$ and $_{\mathbb{P}}E_{i+1}$ have
nonvoid intersection and hence coincide.

First consider parameters $\beta_{k}\rightarrow\alpha_{i},\beta_{k}<\alpha
_{i}$. As stated above, the points $x^{\beta_{k}}$ with $x^{\beta_{k}%
}=g(u^{\beta_{k}})x^{\beta_{k}}$ satisfy $x^{\beta_{k}}\in\mathrm{int}(D_{i})$.

\textbf{Case 1}. $\int_{0}^{\tau_{\alpha_{i}}}\Phi_{u^{\alpha_{i}}}%
(\tau_{\alpha_{i}},s)\left[  Cu^{\alpha_{i}}(s)+d\right]  ds\not \in
\operatorname{Im}(I-\Phi_{u^{\alpha_{i}}}(\tau_{\alpha_{i}},0))$.

Lemma \ref{Lemma_path3} implies that the $x^{\beta_{k}}\in D_{i}$ satisfy%
\begin{equation}
\left\Vert x^{\beta_{k}}\right\Vert \rightarrow\infty\text{ and }%
\frac{x^{\beta_{k}}}{\left\Vert x^{\beta_{k}}\right\Vert }\rightarrow
\mathbf{E}(\Phi^{\alpha_{i}}(\tau_{\alpha_{i}},0);1),\label{7.5c}%
\end{equation}
and by Lemma \ref{Lemma_chains}(i) $e(_{\mathbb{P}}E_{j}^{\hom})\subset
\partial_{\infty}(_{\mathbb{P}}E_{i})$ since%
\begin{equation}
\varnothing\not =\partial_{\infty}(D_{i})\cap e\left(  \pi_{\mathbb{P}%
}\mathbf{E}(\Phi^{\alpha_{i}}(\tau_{\alpha_{i}},0);1)\right)  \subset
\partial_{\infty}(_{\mathbb{P}}E_{i})\cap\,e(_{\mathbb{P}}E_{j}^{\hom
}).\label{7.6}%
\end{equation}

\textbf{Case 2.} $\int_{0}^{\tau_{\alpha_{i}}}\Phi_{u^{\alpha_{i}}}%
(\tau_{\alpha_{i}},s)\left[  Cu^{\alpha_{i}}(s)+d\right]  ds\in
\operatorname{Im}(I-\Phi_{u^{\alpha_{i}}}(\tau_{\alpha_{i}},0))$.

Let, for $k=1,2,\ldots$,%
\[
A_{k}:=I-\Phi_{u^{\beta_{k}}}(\tau_{\beta_{k}},0),\quad b_{k}:=\int_{0}%
^{\tau_{\beta_{k}}}\Phi_{u^{\beta_{k}}}(\tau_{\beta_{k}},s)\left(
Cu^{\beta_{k}}(s)+d\right)  ds.
\]
Then $x^{\beta_{k}}=g(u^{\beta_{k}})x^{\beta_{k}}$ implies $A_{k}x^{\beta_{k}%
}=b_{k}$ and by Proposition \ref{Proposition_periodicODE}(iii), (iv)%
\[
A_{k}\rightarrow A_{0}:=I-\Phi_{u^{\alpha_{i}}}(\tau_{\alpha_{i}}%
,0),\,b_{k}\rightarrow b_{0}:=\int_{0}^{\tau_{\alpha_{i}}}\Phi_{u^{\alpha_{i}%
}}(\tau_{\alpha_{i}},s)\left(  Cu^{\alpha_{i}}(s)+d\right)  ds.
\]
If $x^{\beta_{k}}$ remains bounded, we may assume that $x^{\beta_{k}%
}\rightarrow y^{0}$ for some $y^{0}\in\overline{D_{i}}\subset\mathbb{R}^{n}$
and hence $A_{0}y^{0}=b_{0}$. By Lemma \ref{Lemma_image} there are $x^{k}\in
y^{0}+\mathbf{E}(\Phi_{u^{\alpha_{i}}}(\tau_{\alpha_{i}},0);1)$ with
$\left\Vert x^{k}\right\Vert \rightarrow\infty$ and $\frac{x^{k}}{\left\Vert
x^{k}\right\Vert }\rightarrow\mathbf{E}(\Phi^{\alpha_{i}}(\tau_{\alpha_{i}%
},0);1)$ and again (\ref{7.5c}) follows implying $e(_{\mathbb{P}}E_{j}^{\hom
})\subset\partial_{\infty}(_{\mathbb{P}}E_{i})$.

If $x^{\beta_{k}}$ becomes unbounded, we obtain%
\[
\left\Vert A_{0}\frac{x^{\beta_{k}}}{\left\Vert x^{\beta_{k}}\right\Vert
}\right\Vert \leq\left\Vert A_{0}-A_{k}\right\Vert +\left\Vert A_{k}%
\frac{x^{\beta_{k}}}{\left\Vert x^{\beta_{k}}\right\Vert }\right\Vert
=\left\Vert A_{0}-A_{k}\right\Vert +\frac{\left\Vert b_{k}\right\Vert
}{\left\Vert x^{\beta_{k}}\right\Vert }\rightarrow0,
\]
and (\ref{7.5c}) follows implying $e(_{\mathbb{P}}E_{j}^{\hom})\subset
\partial_{\infty}(_{\mathbb{P}}E_{i})$.

We have shown this inclusion using $\beta_{k}\rightarrow\alpha_{i},\beta
_{k}<\alpha_{i}$. The same arguments can be applied to parameters $\beta
_{k}\rightarrow\alpha_{i},\beta_{k}>\alpha_{i}$, showing that also for the
chain control set $_{\mathbb{P}}E_{i+1}$ the boundary at infinity
$\partial_{\infty}(_{\mathbb{P}}E_{i+1})$ contains the chain control set
$e(_{\mathbb{P}}E_{j}^{\hom})$. This proves the claim and concludes the proof
of the theorem.
\end{proof}

The following controlled linear oscillator illustrates Theorem
\ref{Theorem_main5}.

\begin{example}
\label{Example7.14}Consider the affine control system (cf. \cite[Example
5.17]{ColRS})%
\[
\left(
\begin{array}
[c]{c}%
\dot{x}\\
\dot{y}%
\end{array}
\right)  =\left(
\begin{array}
[c]{cc}%
0 & 1\\
-1-u & -3
\end{array}
\right)  \left(
\begin{array}
[c]{c}%
x\\
y
\end{array}
\right)  +u(t)\left(
\begin{array}
[c]{c}%
0\\
1
\end{array}
\right)  +\left(
\begin{array}
[c]{c}%
0\\
d
\end{array}
\right)  ,\quad u(t)\in\lbrack-\rho,\rho],
\]
where $\rho\in\left(  1,\frac{5}{4}\right)  $ and $d<1$. The equilibria are,
for $u\in\lbrack-\rho,-1)$ and $u\in(-1,\rho]$,%
\[
\mathcal{C}_{1}=\left\{  \left.  \left(
\begin{array}
[c]{c}%
x\\
0
\end{array}
\right)  \right\vert x\in\left[  \frac{d-\rho}{1-\rho},\infty\right)
\right\}  ,\quad\mathcal{C}_{2}=\left\{  \left.  \left(
\begin{array}
[c]{c}%
x\\
0
\end{array}
\right)  \right\vert x\in\left(  -\infty,\frac{d+\rho}{1+\rho}\right]
\right\}  ,
\]
resp. The equilibria in $\mathcal{C}_{1}$ are hyperbolic, since here the
eigenvalues of $A(u)$ are $\lambda_{1}(u)<0<\lambda_{2}(u)$. The equilibria in
$\mathcal{C}_{2}$ are stable nodes since here $\lambda_{1}(u)<\lambda
_{2}(u)<0$. For $u^{0}=-1$ the matrix $A(-1)=\left(
\begin{array}
[c]{cc}%
0 & 1\\
0 & -3
\end{array}
\right)  $ has the eigenvalue $\lambda_{2}(-1)=0$ with eigenspace
$\mathbb{R}\times\{0\}$, hence $e^{A(-1)t}$ has the eigenvalue $1$. There are
control sets $D_{1}\not =D_{2}$ containing the equilibria in $\mathcal{C}_{1}$
and $\mathcal{C}_{2}$, resp., in the interior. For $u^{k}\nearrow u^{0}=-1$,
the equilibria in $D_{1}$ satisfy $(x_{u^{k}},0)\rightarrow(\infty,0)$ and for
$u^{k}\searrow u^{0}=-1$, the equilibria in $D_{2}$ satisfy $(x_{u^{k}%
},0)\rightarrow(-\infty,0)$ for $k\rightarrow\infty$. \ There is a single
chain control set $_{\mathbb{P}}E$ in $\mathbb{P}^{2}$ containing the images
of $D_{1}$ and $D_{2}$ since the eigenspace $\mathbf{E}(e^{A(-1)\tau
};1)=\mathbb{R}\times\{0\}$ satisfies
\[
e(\pi_{\mathbb{P}}\mathbf{E}(e^{A(-1)\tau};1))\subset\partial_{\infty}%
(D_{1})\cap\partial_{\infty}(D_{2})\quad\text{for any }\tau>0.
\]
Concerning the homogeneous part in $\mathbb{P}^{1}$ the projectivized
eigenspace $\pi_{\mathbb{P}}\mathbf{E}(e^{A(-1)\tau};1)$ is contained in the
invariant control set $_{\mathbb{P}}D_{2}^{\hom}=\pi_{\mathbb{P}}\{\left.
(x,\lambda_{2}(u)x)^{\top}\right\vert x\not =0,u\in\lbrack-\rho,\rho]\}$ and
$_{\mathbb{P}}D_{2}^{\hom}$ is the projection of a control set $_{\mathbb{R}%
}D_{2}^{\hom}$ in $\mathbb{R}^{2}$, since $0\in\mathrm{int}(\Sigma
_{Fl}(_{\mathbb{P}}D_{2}^{\hom}))$.
\end{example}

\section{Appendix}

This appendix presents the proof of Proposition \ref{Proposition_periodicODE}.
Assertion (i) follows from the variation-of-para\-meters formula and (ii) is a
consequence of (i).

(iii) The principal fundamental solutions $\Phi^{k}(t,s)$ satisfy for
$t,s\in\lbrack0,\tau_{0}+1]$%
\[
\left\Vert \Phi^{k}(t,s)\right\Vert \leq1+\left\vert \int_{s}^{t}\left\Vert
P^{k}(\sigma)\right\Vert \left\Vert \Phi^{k}(\sigma,s)\right\Vert
d\sigma\right\vert .
\]
By the generalized Gronwall inequality (cf. Amann \cite[Lemma 6.1]{Amann}) it
follows that $c_{1}=\sup_{k\in\mathbb{N},\,\,t,s\in\lbrack0,\tau_{0}%
+1]}\left\Vert \Phi^{k}(\,t,s)\right\Vert <\infty$. For all $\,t,s\in
\lbrack0,\tau_{0}+1]$%
\begin{align*}
&  \left\Vert \Phi^{k}(t,s)-\Phi^{0}(t,s)\right\Vert =\left\Vert \int_{s}%
^{t}\left[  P^{k}(\sigma)\Phi^{k}(\sigma,s)-P^{0}(\sigma)\Phi^{0}%
(\sigma,s)\right]  d\sigma\right\Vert \\
&  \leq\left\Vert \int_{s}^{\,t}\left[  P^{k}(\sigma)-P^{0}(\sigma)\right]
\Phi^{k}(\sigma,s)d\sigma\right\Vert +\left\vert \int_{s}^{\,t}\left\Vert
P^{0}(\sigma)\right\Vert \left\Vert \Phi^{k}(\sigma,s)-\Phi^{0}(\sigma
,s)\right\Vert d\sigma\right\vert .
\end{align*}
The first term is bounded by $c_{k}:=c_{1}\int_{0}^{\tau_{0}+1}\left\Vert
P^{k}(\sigma)-P^{0}(\sigma)\right\Vert d\sigma\rightarrow0$ for $k\rightarrow
\infty$. Again by \cite[Lemma 6.1]{Amann} it follows that for $\,t,s\in
\lbrack0,\tau_{0}+1]$%
\[
\left\Vert \Phi^{k}(t,s)-\Phi^{0}(t,s)\right\Vert \leq c_{k}+\left\vert
\int_{s}^{\,t}c_{k}\left\Vert P^{0}(r)\right\Vert \exp\left\vert \int
_{r}^{\,t}\left\Vert P^{0}(\sigma)\right\Vert d\sigma\right\vert dr\right\vert
.
\]
The right hand side converges to $0$ uniformly in $\,t,s$ for $k\rightarrow
\infty$ since $c_{k}\rightarrow0$.

(iv) The assumption implies that there are unique $\tau_{k}$-periodic
solutions given by%
\[
x^{k}:=(I-\Phi^{k}(\tau_{k},0))^{-1}\int_{0}^{\tau_{k}}\Phi^{k}(\tau
_{k},s)z^{k}(s)ds.
\]
Then $\left\Vert \int_{\tau_{0}}^{\tau_{k}}\Phi^{k}(\tau_{k},s)z^{k}%
(s)ds\right\Vert \rightarrow0$ since $\tau_{k}\rightarrow\tau_{0}$ and the
integrands are uniformly bounded, and
\begin{align*}
&  \left\Vert \int_{0}^{\tau_{0}}\left[  \Phi^{k}(\tau_{k},s)z^{k}(s)-\Phi
^{0}(\tau_{0},s)z^{0}(s)\right]  ds\right\Vert \\
&  \leq\left\Vert \int_{0}^{\tau_{0}}\left[  \Phi^{k}(\tau_{k},s)-\Phi
^{0}(\tau_{0},s)\right]  z^{k}(s)ds\right\Vert +\left\Vert \int_{0}^{\tau_{0}%
}\Phi^{0}(\tau_{0},s)\left[  z^{k}(s)-z^{0}(s)\right]  ds\right\Vert \\
&  \leq\left[  \sup_{s\in\lbrack0,\tau_{0}]}\left\Vert \Phi^{k}(\tau
_{k},s)-\Phi^{0}(\tau_{k},s)\right\Vert +\sup_{s\in\lbrack0,\tau_{0}%
]}\left\Vert \Phi^{0}(\tau_{k},s)-\Phi^{0}(\tau_{0},s)\right\Vert \right]
\int_{0}^{\tau_{0}}\left\Vert z^{k}(s)\right\Vert ds\\
&  \qquad+\sup_{s\in\lbrack0,\tau_{0}]}\left\Vert \Phi^{0}(\tau_{0}%
,s)\right\Vert \int_{0}^{\tau_{0}}\left\Vert z^{k}(s)-z^{0}(s)\right\Vert ds.
\end{align*}
This converges to $0$ and it follows that $x^{k}\rightarrow x^{0}$.

\textbf{Acknowledgements. }\ We thank Luiz A.B. San Martin and Adriano Da
Silva for valuable suggestions and comments. We are also indebted to anonymous
reviewers of an earlier version whose significant queries helped to improve
the paper.

\end{document}